\documentclass[a4paper,11pt]{amsart}
\usepackage{amssymb}
\usepackage{amsthm}
\usepackage{hyperref}
\usepackage[mathscr]{euscript}

\newtheorem{theorem}{Theorem}[section]
\newtheorem{lemma}[theorem]{Lemma}
\newtheorem{claim}{Claim}
\newtheorem{fact}[theorem]{Fact}
\newtheorem{proposition}[theorem]{Proposition}
\newtheorem{corollary}[theorem]{Corollary}

\theoremstyle{definition} 

\newtheorem{definition}[theorem]{Definition}
\newtheorem{example}{Example}
\newtheorem{counterexample}[example]{Counterexample}
\newtheorem{examples}{Examples}
\newtheorem{remarks}[theorem]{Remarks}
\newtheorem{remark}[theorem]{Remark}

\newtheorem{open}{Open Question}

\newcommand{\cx}{\mbox{\bf cx}}
\renewcommand{\c}{\mbox{\bf c}}
\newcommand{\sx}{{\bf sx}}
\newcommand{\chief}{{\bf chief}}
\newcommand{\Ax}{{\bf Axioms}}
\newcommand{\Fit}{{\bf Fit}}
\newcommand{\Der}{{\bf Der}}
\newcommand{\Solv}{{\bf Solv}}
\newcommand{\Nil}{{\bf Nil}}
\newcommand{\solv}{{\bf solv}}
\newcommand{\JH}{{\bf JH}}
\newcommand{\soc}{\mbox{\rm soc}}
\newcommand{\cxsolv}{\cx_{solv}}
\newcommand{\Z}{\mathbb{Z}}
\newcommand{\sur}{\twoheadrightarrow}

\newcommand{\N}{\mathbb{N}}
\newcommand{\1}{{\mathbf 1}}
\renewcommand{\S}{{\mathscr{S}}}
\newcommand{\SNAG}{{\mathscr{SNAG}}}
\newcommand{\SIMPLE}{{\mathscr{SIMPLE}}}
\newcommand{\surS}{{\bf sur}}
\newcommand{\subS}{{\bf sub}}
\newcommand{\V}{{\bf V}}
\newcommand{\n}{{\bf n}}
\newcommand{\Span}{Span}

    \title{Hierarchical Complexity of Finite Groups}

    \author{Chrystopher L. Nehaniv}
    \thanks{This work was supported by the Natural Sciences and Engineering Research Council of Canada (NSERC), funding reference number RGPIN-2019-04669. 
Cette recherche a \'et\'e financ\'ee par le Conseil de recherches en sciences naturelles et en g\'enie du Canada (CRSNG), num\'ero de r\'ef\'erence RGPIN-2019-04669.\\ The author also is grateful to Mr.\ Thomas George, recipient of a University of Waterloo President's Award as undergraduate research assistant, for assistance with constructing software tools with which the counterxample groups in this paper were initially found and studied by the author.}

    \address{University of Waterloo,  Ontario, Canada }
    \email{chrystopher.nehaniv@uwaterloo.ca}

\begin{document}

    \maketitle
 \begin{quote}
   {\small  {\it  In any field of mathematics, complexity is first level of sophistication beyond knowing the building blocks.} 
     --John Rhodes}
 \end{quote}
\begin{abstract} 
 What are simplest ways to construct a finite group from its atomic constituents?   To understand part-whole relations between finite simple groups (``atoms'') and the global structure of finite groups, we axiomatize complexity measures on finite groups.  
From the Jordan-Hölder theorem and Frobenius-Lagrange embedding in an iterated wreath product, any finite group $G$ can be constructed from a unique collection of simple groups, its Jordan-H\"{o}lder factors, each with well-defined multiplicities through iterated extension by simple groups.  What is the least number of levels needed in such a hierarchical construction if a level is allowed to include several of these atomic pieces? 
To pose and answer this question rigorously, we give a natural set of hierarchical complexity axioms for finite groups relating to constructability, extension, quotients, and products, and prove these axioms are satisfied by a unique maximal complexity function $\cx$. 
We prove this function is the same as the minimal number of ``spans of gems'' (direct products of simple groups) in a subnormal series with all factors of this type. Hierarchical complexity $\cx$ is thus effectively computable, and bounded below by all other complexity measures satisfying the axioms, including generalizations of derived length, Fitting height and solvability.  Also, the hierarchical complexity of a normal subgroup is bounded above by the complexity of the whole group, although this is not assumed in the axioms and does not follow from them.

For solvable groups, the unique maximal group complexity measure defined satisfying the axioms agrees with the restriction of the one for all finite groups, and in addition satisfies an embedding axiom - which decidedly cannot be applied in the general case of all finite groups.  In both cases, the complexity of a group is bounded above and below by various natural functions. In particular, hierarchical complexity is sharply bounded above by socle length, which yields a canonical decomposition and satisfies all the axioms except the extension axiom. 
 Examples illustrate applications of the bounds and axiomatic methods in determining complexity of groups. We show also that minimal decompositions need not be unique in terms of what components occur nor their ordering. 

The complexity axioms are also shown to be independent. 
\end{abstract}

\newpage

\tableofcontents
\newpage

\section{Introduction}\label{sec_intro}

\subsection{Motivation and Brief Overview}
This article follows the viewpoint of Rhodes~\cite{wildbook} on complexity of mathematical objects, developing it for finite groups using a global, axiomatic approach.
Such an axiomatic approach complements the study of extensions of finite groups initiated by Otto H\"older in the 19th century which attends to the local process of putting together groups from simpler parts.
Such complexity theories develop the study of part-whole relations and provide a global hierarchical picture for understanding how simple building block components of
 objects may be put together to constituent the whole (albeit perhaps non-uniquely).  
Previously such an axiomatic complexity theory has been developed by Krohn, Rhodes, Tilson, Eilenberg, Steinberg and many others
for  finite semigroups and finite automata  (\cite{KrohnRhodes-PNAS}, \cite{EilenbergVolB}, \cite{TilsonInEilenbergB}, \cite{BioCpx},\cite[Ch.~4]{qtheory})  and for finite aperiodic semigroups~\cite{aperiodicCpx}. 
Remarkably, although these areas are close to finite group theory and their axioms are abstracted from the research practices of group theory \cite[Ch.~2-3]{wildbook}, such a complexity theory for all finite groups has not been developed.\footnote{The treatment of subgroups (or equivalently embedding) presents a subtle point (cf.~Section~\ref{SubgroupSec} for this ``Michelangelo Problem''), that is not an issue for solvable groups. It turns out we can take an agnostic approach that avoids any axiom mentioning the complexity of subgroups or complexity of normal subgroups and still develop an axiomatic theory that reveals how to handle this issue.  As a consequence of this approach, it naturally follows that maximal complexity function and other measures satisfy the Normal property, i.e., that a normal subgroup has complexity no more than the complexity of the whole -- although, as we show, this statement too is independent of the other axioms!} 
An appropriate formulation of the complexity axioms in the context of finite group theory allows us to do so here, at least some intitial steps.  

We axiomatize hierarchical complexity of finite group and study the functions satisfy the axioms and subsets of them.  The axioms are shown to be independent and there is a unique maximal complexity function $\cx$ on finite groups satisfying them.  Certain well-known functions are also complexity functions on finite or finite solvable groups (Fitting height, derived length and others), which are lower bounds to $\cx$. 
The socle length is closely related to $\cx$, but shown to be a sharp upper bound and fails one of the axioms. \\

\subsection{Complexity Axioms for Finite Groups}\label{CpxAxioms}
We consider functions $\c$ from the class of finite groups to the natural numbers satisfying the following natural complexity axioms\footnote{All groups considered here are finite. } :

\begin{enumerate}
\item Extension.  If   $N$ is normal in $G$, then $\c(G) \leq \c(N) + \c(G/N)$.
\item Quotient. 
If a homomorphism maps $G$ onto $H$, then $\c(G)\geq \c(H)$.
\item Normal.  If $N$ is normal in $G$, $\c(N) \leq \c(G)$.
\item Constructability from building blocks.   Each group $G$ either has complexity at most 1 or can be constructed by iterated extension from groups with complexity at most 1.   

\item  Product.    If $G= H \times K$,  then $\c(G)=\max({\c(H), \c(K)})$.
\item Initial Condition. The complexity of the trivial group $\1$ is 0, i.e. $\c(\1)=0$

\end{enumerate}

That ``the constituent parts are not more complex than the whole'' is captured by the quotient and normal properties. 

We say a group $G$ is an {\em extension of $Q$ by $N$} if $N$ is isomorphic to a normal subgroup of  $G$ and $G/N$ is isomorphic to $Q$.
The extension axiom expresses that the complexity of group is bounded by a simple function of the complexity of  component parts of the group. 

Since isomorphic groups map onto each other, any $\c$ satisfying the quotient axiom must assign them the same value.
So we may regard such a $\c$ as a function on the set of isomorphism classes of finite groups.  The quotient axioms says that a homomorphic image of a group is no more complex than the group - this is also called ``the mapping axiom''. 

Another immediate consequence
of the quotient axiom is:

(2$^\prime$). If $N$ is normal in $G$, $\c(G/N) \leq \c(G) $. \\
For functions that are well-defined on isomorphism classes of finite groups, (2$^\prime$) is equivalent to the quotient axiom (2) since for $\varphi: G\sur H$,  by the first isomorphism theorem $H\cong G/N$ for $N=\ker \varphi$.  Hence, (2$^\prime$) plus well-definedness on isomorphism classes may be used interchangably with (2). 

The product axiom captures the fact that direct product of factors does not combine them in a hierarchical way, but in parallel,  which is not more complex than the independent factors.

Precisely, the constructability axiom says each $G$ is {\em constructable from groups of complexity at most 1}, i.e., either $\c(G)\leq 1$ or $G$ is an extension of of $Q$ by $N$ where $Q$ and $N$ are constructable from groups of complexity at most 1.

Recall that if $G$ is an extension of $Q$ by $N$ then $G$ is isomorphic to a subgroup of the wreath product $N\wr Q$ which acts on the set $N\times Q$  (e.g., Frobenius-Lagrange coordinates \cite[pp.~29--33]{baumslag},  \cite[Theorem~1.17]{AutomataNetworks},  Kalu\v{z}hnin-Krasner Universal Embedding Theorem \cite{KrasnerKaloujnine}, \cite[Theorem 2.6A]{DixonMortimer}).\footnote{Note that 
we follow  \cite{wildbook,qtheory,AutomataNetworks} for basic results and notation for wreath products: 
Recall the (permutational) wreath product $W= N\wr Q$ is the group of mappings on $N\wr Q$ of the form $w=(w_2,w_1)$ where
$(n,q)w=(n w_2(q), q w_1)$ for some $w_2: Q\rightarrow N$ and $w_1\in Q$. Then $W$ is a semidirect product $N^Q \rtimes Q$.
Similarly, the $n$-fold wreath product of $G_n, \ldots, G_1$ is $W= G_n \wr \cdots \wr G_1$,  consisting of functions
$w$ on the set $G_n\times \cdots \times G_1$ of the form $w=(w_n, \ldots, w_1)$ with $w_i: G_{i-1}\times \cdots \times G_1 \rightarrow G_i$ ($1\leq i \leq n)$, so  $w_1\in G_1$, with  $(g_n, \ldots , g_1) w= (g_n w_n(g_{n-1},\ldots, g_2 w_2(g_1), g_1 w_1)$ ($g_i \in G_i$).  
More generally, if $G_i \leq Sym(X_i)$, the symmetric group on set $X_i$, we write
$(X_n, G) \wr \cdots \wr (X_1, G_1)$ for the permutation group $W$ acting on $X_n\times \cdots \times X_1$ consisting of all $w=(w_n, \ldots, w_1)$ with $w_i: X_{i-1}\times \cdots \times X_1 \rightarrow G_i$ ($1\leq i \leq n)$,  $w_1\in G_1$, with  
$$(x_n, \ldots , x_1) w= (x_n w_n(x_{n-1},\ldots, x_2 w_2(g_1), x_1 w_1),$$ $x_i\in X_i, g_i\in G_i$. Thus $G_n \wr \cdots \wr G_1$ denotes $(G_n,G_n) \wr \cdots \wr (G_1,G_1)$, where each $G_i$ acts on itself by right multiplication.}

Therefore,  the \underline{constructability axiom} may be restated as follows:\\

(4$^\prime$) Each group $G$ has a subnormal series 
$$\1 =G_n \lhd G_{n-1}  \lhd\cdots \lhd G_0=G, \mbox{  with  $\c(G_{i}/G_{i+1}) \leq 1$},$$
 for each $0\leq i < n$ and some $n\geq 0$.\\
 
\noindent
\underline{Remarks and Questions on Global Structure}: As the iterated (permutational) wreath product is associative, we may `parenthesize' the quotients $G_{i+1}/G_i$ ($0 \leq i < n$), corresponding to a subnormal series for a given group $G$ arbitrarily in constructing the embedding of $G$ in the iterated wreath product of these quotients:  $$G \leq G_{n-1}/G_{n} \wr \cdots \wr G_0/G_{1}.$$
Observe that the Jordan-H\"older factors of $G$ are the union of the Jordan-H\"older factors for $G_{i}/G_{i+1}$ with the same multiplicities. If each quotient $G_{i}/G_{i+1}$ has complexity at most 1, 
then $n$  bounds the complexity of $G$  above (by the extension axiom).
{\em Therefore, determining complexity of $G$ is solved if we determine how to `pack' the simple pieces of $G$ (its Jordan-H\"older factors) into these quotient groups  with a shortest hierarchical decomposition using the wreath product, giving the global structure of how the parts can be put together to build the whole.}\footnote{The embedding a group   using iterated extensions into the wreath product constitutes a so-called a `cascade product' of the factors synthesizing the group from these constituents~\cite{cascadeproduct} and is supported by computer algebraic tools of the {\sc SgpDec} package for {\sc GAP}~\cite{SgpDec}.}

Is such a minimal decomposition of $G$ into constituent components canonical? 
Are the factors $G_{i}/G_{i+1}$ constituting $G$ for a minimal decomposition unique (up to isomorphism)?  Is their order and manner of putting them unique?   We shall see that for all finite groups there is a canonical 
decomposition via the socle series that sharply bounds complexity above, but that for hierarchical complexity itself neither the constituents nor their ordering are unique, as we show by examples. 
On the other hand, for solvable groups, both the Fitting series and derived series give lower bounds with canonical constituents.\\

\noindent
\underline{Remarks on Axioms, Independence and Bounds}:
Hierarchical complexity will be constructed and shown to be the unique maximal complexity function satisfying the axioms. 
Therefore by constructing many complexity functions and functions the fail a single axiom or property, we shall be able to use them to show lower and upper bounds, respectively, for the hierarchical complexity on all groups and particular classes of groups as well as in concrete cases of computing complexity. This will also allow us to show that the complexity axioms themselves together with the normal subgroup property are logically independent of one another. 

Moreover, the study and systematization of the properties of many well-known functions (such as nilpotency class,  derived length, Fitting height, socle length, minimal number of generators, exponent,  etc., - sometimes called {\em arithmetic functions}) on groups and subclasses of groups is achieved by placing them in the context of the  complexity axioms allowing  to be related and compared in a larger framework.

\begin{lemma}\label{simplecpx}
Let $S$ be a finite simple group and $\c$
 a function on finite groups that satisfies the constructability axiom. 
 Then $\c(S)\leq 1$. 
 \end{lemma}
 \begin{proof} This follows from (4$^\prime$), since each simple group has a unique subnormal series of length 1. \end{proof}
Conversely, 
\begin{lemma}\label{constru->simple1}
If a function on finite groups $\c$ satisfies $\c(S)\leq 1$ for every simple group $S$, then
$\c$ satisfies the constructability axiom. 
\end{lemma} \begin{proof} 
Since every finite group $G$ is either simple or has a maximal normal subgroup, $G$ can be constructed by iterated extension from finite simple groups.  In particular, $G$ is isomorphic to a subgroup of the iterated wreath product of its Jordan-H\"older factors (e.g.\ \cite[Theorem 1.17]{AutomataNetworks}). Alternatively, since each of the quotients  $S$ in a composition series for $G$ is a simple group,
constructability in the form (4$^\prime$) follows from the hypothesis that $\c(S)\leq 1$ for each of these. \end{proof}
It is now immediate that 
the constructability axiom is equivalent to the complexity of simple groups being bounded above by 1:
\begin{lemma}\label{construct-simple}
Let $\c$ be a function from finite groups to $\N$.  Then:\\
$\c$ satisfies  the constructability axiom if and only if $\c(S)\leq 1$ for all $S$ simple. 
\end{lemma}

\noindent
{\em{\underline{\bf Important Note:}} In the sequel, we shall \textbf{not} assume the Normal property~(3) hold for our complexity functions. Nevertheless, it will be shown to hold for the unique maximal complexity function $\cx$ on finite groups and related measures. }

\section{The Unique Maximal Complexity Function}

By the constructability axiom,  each simple group has complexity at most one. 
The following is clear.
\begin{lemma}\label{trivial}
(1) The trivial complexity function $z$ that assigns to each group the number zero satisfies the complexity axioms.\\
(2) 
The function $\delta$ with $\delta(1)=0$ for the trivial group, and $\delta(G)=1$ for $G$ a nontrivial finite group satisfies the complexity axioms.
\end{lemma}
So complexity functions exist.

A complexity function $\c$ is {\em maximal} if it pointwise dominates any other one $\c'$ satisfying the axioms, i.e.,
$\c(G) \geq \c'(G)$ for all finite groups $G$.   If $\c$ and $\c'$ are both maximal, then $\c(G) \geq \c'(G)$ and $\c'(G) \geq \c(G)$, whence $\c=\c'$. Thus, if there is a maximal complexity function, it is unique.

\subsection{The Jordan-Hölder Function}

Let $\JH(G)$ be the number of Jordan-Hölder factors of a finite group $G$ counting multiplicities.  

\begin{lemma} \label{JH}
\begin{enumerate}
\item $\JH$ satisfies the constructability, quotient, normal subgroup, initial condition, and extension axioms, but not the product axiom.     
\item If $\c$ satisfies the axioms except possibly the product axiom, then $\c(G)\leq \JH(G)$ for any finite group. \label{DomByJH} 
\item $\JH$ is the unique maximal function on groups satisfying the axioms except possibly  the product axiom. 
\end{enumerate}
\end{lemma}
\begin{proof}  (1) Since the number of Jordan-Hölder factors of a simple group is 1 and every finite group can be constructed from simple groups by iterated extension, constructability holds for $\JH$. 
If $G$ is an extension of $Q$ by $N$, then since a composition series for $G$ can be constructed from the one from $Q$ by multiply each subgroup in the series by $N$, and followed by a composition series for $N$ down to the trivial group.  The quotients of this series are exactly the Jordan-Hölder factors of $Q$ together with those of $N$ counting mulitplicities.
Therefore
$\JH(G)=\JH(N)+\JH(Q)$, whence extension holds for $\JH$. 
The quotient axiom $\JH(G/N) \leq \JH(G)$ holds since quotienting removes the occurrences of Jordan-Hölder factors of $N$ from  those of $G$. 
The product axiom does not hold since, e.g. $\JH(K \times H)=2$ but $\max(\JH(K),\JH(H))=1$ for any simple groups $K$ and $H$. Since the trivial group has no Jordan-Hölder factors $\JH(\1)=0$.

(2) Consider functions $\c$ on all finite groups which satisfy all axioms except possibly the direct product axiom. 
Let $$\epsilon(G)= \sup \{\c(G) : \c \mbox{ satisfies the complexity axioms, except possibly the product axiom}\}.$$
Suppose $\c$ satisfies the axioms except possibly the direct product axiom. 
We claim $\c(G)\leq \JH(G)$ for every finite group $G$, showing this by induction on the length of a composition series for $G$, i.e., on $\JH(G)$.  The case $\JH(G)=0$ is trivial.
If $\JH(G)=1$, then $G$ is simple, and $G$ cannot be constructed by iterated extension from other groups, so $\c(G)\leq 1$ by constructability, and the claim holds. 
If $\JH(G)>1$, then $G$ is not simple
and $G$ has a maximal normal subgroup $N$ with $G/N$ is simple, so $\JH(G/N)=1$. By the extension axiom, $\c(G) \leq \c(N) + \c(G/N)$. Now $\c(N) \leq \JH(N)$ by induction hypothesis,  so $\c(G)\leq \JH(N)+1=\JH(G)$.
We conclude that $\c(G) \leq \JH(G)$ for all $G$ by induction.  Since $\c$ was arbitrary, 
 $\epsilon(G) \leq \JH(G)$.
Since $\JH$ itself satisfies the axioms (except the product axiom), it follows that $\epsilon(G)=\JH(G)$ for all $G$. This proves (3).
\end{proof}

\subsection{Existence of a Maximal Group Complexity Function}

\begin{theorem}\label{uniqueCpx}
There is exists a unique maximal hierarchical complexity function $\cx$ on finite groups satisfying the axioms, where 
$$\cx(G)= \sup \{\c(G) : \c \mbox{ satisfies the complexity axioms }\}. $$

\end{theorem}
\begin{proof} 
By Lemma~\ref{trivial}, the supremum is over a non-empty set of natural numbers. By Lemma~\ref{JH}(\ref{DomByJH}),  $\JH(G)$ is an upper bound on $\c(G)$ for each complexity function $\c$. Therefore $\cx(G)$ is natural number no larger than $\JH(G)$.

We claim $\cx$ satisfies the complexity axioms.

Given $G$ choose $\c$ satisfying the axioms such that
$\c(G)$ is maximal,  then for any $N$ normal in $G$, we have $$\cx(G)=\c(G) \leq \c(N)+\c(G/N) \leq \cx(N) + \cx(G/N).$$
The first inequality holds since the axiom extension holds for $\c$.  The second holds by definition of $\cx$. Therefore $\cx$ satisfies  extension.

For the quotient axiom, choose a complexity function $\c_Q$ for which $\c_Q(G/N)$ is greatest, then by the quotient axiom for $\c_Q$ and definition of $\cx$
$$\cx(G) \geq \c_Q(G) \geq \c_Q(G/N)=\cx(G/N).$$

Clearly, $\cx(\1)=0$ since $\c(\1)=0$ for each $\c$ satisfying the axioms.
For each finite simple group $K$, we have $\cx(K)\leq 1$ since for each complexity function $\c(K)\leq 1$.  Hence, each group
$G$ can be constructed by iterated extension from groups $K$ for which $\cx(K)\leq 1$. Thus constructability axiom holds for $\cx$. (In fact, $\cx(K)\geq \delta(K)=1$ for simple $K$ shows $\cx(K)=1$.)

For the product axiom, suppose
$G=H\times K$, and
choose a complexity function $\c$ for which $\c(G)$ is greatest,  $\c'$ for which $\c'(H)$ is greatest, and $\c''$ for which $\c''(K)$ is greatest. Then we have
\begin{eqnarray*}
\cx(G)&= &\c(G) \mbox{ by choice of $\c$}\\
         &= &\max\{\c(H),\c(K)\} \mbox{ by product axiom for $\c$}\\
         &\leq&\max\{\c'(H),\c''(K)\} \mbox{ by choice of $\c'$ and $\c''$}\\
         & = &\max\{\cx(H),\cx(K)\} \mbox{ by choice of $\c$ and $\c''$ and definition of $\cx$}
         \end{eqnarray*}
         On the other hand,
 \begin{eqnarray*}   
 \cx(G) & \geq &\max\{\c'(G), \c''(G)\} \mbox{ since $\cx$ is maximal at $G$ }\\
         & =&\max\{\max\{\c'(H),\c'(K)\},
                  \max\{\c''(H),\c''(K)\}\} \mbox{ by the product axiom for $\c'$ and $\c''$}\\
         & \geq &\max\{\c'(H), \c''(K)\} \\
         & = & \max\{\cx(H),\cx(K)\} \mbox{ by choice of   $\c'$ and $\c''$ and definition of $\cx$ at $H$ and at $K$}.
       \end{eqnarray*}   
This proves $\cx(H\times K)=\max\{\cx(H),\cx(K)\}$ as required.

By its definition $\cx(G)\geq \c(G)$ for any $\c$ that also satisfies the axioms. It follows $\cx$ is unique, since if $\c$ were also pointwise maximal then $\cx(G)\geq \c'(G) \geq \cx(G)$, and so $\cx(G)=\c(G)$ for every group $G$.  
\end{proof}

From the argument of the proof, the following is clear.
\begin{corollary}
The function given by the pointwise maximum of any set of complexity functions is complexity function.  That is, 
let $I$ index a set of complexity functions, then the function $$\c_I(G)=\max_{i\in I} \c_i(G),$$
for  $G$ a finite group,   satisfies the complexity axioms. 
\end{corollary}

\section{The Subgroup Axiom (or ``Embedding Axiom'')}\label{SubgroupSec}

\begin{quote}
 {\small \it   
The sculpture is already complete within the marble block, before I start my work. It is already there, I just have to chisel away the superfluous material.}  \ \ \ \ \ --Michelangelo
\end{quote}
\vspace{1em}

It is possible to consider another axiom, namely that the complexity of a subgroup should be no more than the complexity of a group containing it:\\[.5ex]

\underline{Subgroup Axiom.}  If $H$ is a subgroup of $G$, then $\c(H)\leq \c(G)$.\\[.5ex]

\noindent
While this axiom sounds reasonable, it leads to mainly uninteresting complexity functions reminiscent of detecting only the marble block in which a beautiful sculpture might be found. 

\begin{proposition}[``The Michelangelo Problem'']\label{cpx-no-subgroup}
If a complexity measure $\c$ satisfies the subgroup axiom, its value is bounded above by 1.
\end{proposition}
\begin{proof}
By the next lemma, any finite group $G$ embeds into a finite simple group $S$.   
This simple group has complexity at most $1$  by Lemma~\ref{simplecpx}, whence the subgroup axiom implies
$\c(G) \leq \c(S) \leq 1$.
\end{proof}

\begin{lemma}\label{EmbedInAlternatingGroup}
Every finite group $G$ embeds in a finite simple alternating group.
\end{lemma}
\begin{proof}  If $|G|\leq 3$ then $G$ embeds in the smallest simple alternating group $A_5$.
If $|G|$ is at least 3, then consider the right regular representation of $G$ acting on itself.
This gives an embedding of $G$ in the symmetric group on $|G|$ points. From this we can construct
an embedding of $G$ into an alternating group on $n=2|G|$ points, which is a simple group since $n\geq 5$ as
shown by Galois. Taking two disjoint copies of 
$G$ as a set, let $G$ act on each one as in the right regular representation within each copy by right multiplication.
Whether $g\in G$ acts on the itself by an odd or even permutation, the action of $g$ on the $2|G|$  points is always even since
we have two disjoint copies of each cycle  in $g$ action on $|G|$.  
This gives an embedding of $G$ in the simple group.
\end{proof}

We shall see in section~\ref{solv} that for complexity measures on solvable groups the subgroup axiom does not lead to the triviality seen in ``Michelangelo problem'' but is useful (and actually a consequence for the hierarchical complexity function and related functions on solvable groups).

\section{Spans of Gems and Hierarchical Complexity}

Viewing finite simple groups as the rare `gems' among finite groups, from which all finite groups are built,
call a finite group $G$ {\em a span of gems} if it is simple or the direct product of simple groups, i.e.,
$$G\cong K_1 \times \cdots \times K_n \mbox{ (with each $K_i$ simple, $0 \leq i \leq n$)}.$$
(Note: $n=0$, $G$ is the trivial group $\1$.)
We also refer to such a group as a {\em m\={a}l\={a}}\ (or  {\em necklace) of simple groups}\footnote{This is a Sanskrit word referring to a circular chain of beads in the form of a rosary.} in line with terminology introduced for other complexity functions (cf.~Section~\ref{der-fit-solv-cpx}).

Indeed, from the constructability and product axioms, it follows that a span of gems can have complexity at most one. 
We shall see that for the maximal complexity function $\cx$ on finite groups, the nontrivial spans of gems are exactly the groups of complexity of 1, and layering them leads to all higher complexity groups. 

We first investigate spans of gems in Section~\ref{1stSplittingLemmas},  showing in detail how these groups are split as internal direct products by their normal subgroups, and also by their quotients.  In Section~\ref{sec-mu-complexity}, Theorem~\ref{mucpx} characterizes the maximal complexity function on groups $\cx$ as the hierarchical complexity measure $\mu$ where spans of gems comprise the `layers' in a hierarchy in a minimal length decomposition of group $G$, with $\cx(G)$-many layers in the recursive construction of finite groups from simpler ones by iterated extensions.

\noindent
\underline{Notation}: For  a collection of groups $\S$, let $\Span(\S)$ denote all groups isomorphic to direct products $G_1\times \cdots \times  G_n$ with each $G_i\in \S$ ($0 \leq i \leq n$).  We say $G$ is a {\em span of groups from $\S$} (or  a {\em m\={a}l\={a} of $\S$ groups})  if $G\in \Span(\S)$.

\subsection{Spans of Gems:  Splitting Lemmas~I}\label{1stSplittingLemmas}

In this subsection we prove two variants of ``the span of gems splitting lemma''.  Namely, we shall see that they are split by any normal subgroup in complementary spans of gems in a particular manner that is ``rigid'' for their simple non-abelian group (SNAG) factors but ``fluid'' for their maximal elementary abelian factors (direct powers of prime-order cyclic groups).  These results reformulating classical facts (see references in the first paragraph of Section~\ref{MinSocHier})  are needed in the sequel.

We begin with this observation: 

\begin{lemma}[SNAG Factor Rigidity]\label{SNAGrigidity}
Let $H$ be a simple non-abelian group (SNAG). Consider any direct product $H\times K$ with another group $K$. Let $\pi: H\times K \rightarrow K$ be the projection morphism. 
For $N \lhd H\times K$, if $(h_0, k_0) \in N$ with $h_0\neq 1$ then $H\times \1 \lhd N$, and
 $N=H \times \pi(N)$. Otherwise $N=\1 \times \pi(N)$. In either case, $\pi(N)\lhd K$.
\end{lemma}
\begin{proof} If there is no $(h_0,k_0)$ with $h_0$ a non-trivial element of $H$, then clearly $N=\1 \times \pi(K)$, and also  $\pi(N) \lhd K$ since $\1\times \pi(N)=N$ is normal in $H\times K$. Thus, the assertion holds.

Otherwise there is a $(h_0,k_0)\in N$  with $h_0 \neq 1$, so also $(h_0^{-1}, k_0^{-1})\in N$. Since $N$ is normal, for all $(h,k)\in H\times K$,
 $(h^{-1}h_0^{-1}h, k^{-1}k_0^{-1}k) \in N$. So taking $k=k_0$, and multiplying the following two members of $N$ yields
$$(h^{-1}h_0^{-1}h, k_0^{-1}k_0^{-1}k_0)(h_0, k_0)=(h^{-1}h_0^{-1}h h_0,1) \in N.$$
Since $H$ is non-abelian, there is an $h\in H$ so $h^{-1}h_0^{-1}h h_0$ is not $1$. Write $h_1=h^{-1}h_0^{-1}h h_0$.  Since $h_1\neq 1$ and $H$ is simple, the normal closure in $H\times K$ of the group generated by $(h_1,1) \in N$ contains $H\times \1$. Thus, $H\times\1 \leq N$. Of course $H\times \1$ is normal in $N$ since it is normal in $H\times K$.  
Now for any $(h,k) \in N$, also $(1,k)\in N$ since $(h^{-1}, 1) \in H\times \1 \leq N$ and $(1,k)=(h,k)(h^{-1},1)$.  That is $\1 \times \pi(N) \leq N$. 
Hence we can write each $(h,k)$ in $N$ as a member of $H\times \1$ times a member of $\1\times \pi(N)$. By normality of $N$, any conjugate of $(1,k)$ lies in $N$, so $\pi(N)\lhd K$ (since the $H$-coordinate of such a conjugate is $1$).  Thus $N$ is the (internal) direct product of $H\times \1$ and 
$\1 \times \pi(N) \lhd N$. 
\end{proof}
Thus if a normal subgroup intersects a SNAG factor of a direct product nontrivially, it contains the whole SNAG. By iterated the Lemma, we have, 
\begin{corollary}
A normal subgroup $N$ of a direct product of SNAGs and another group $K$ is the direct product of a subset of these SNAGs and the projection $\pi(K)$ of $N$ to $K$, and, moreover, $\pi(N)\lhd K$.
\end{corollary}

 The `span of gems' groups play a special role in the hierarchical complexity theory of finite groups and have special properties.
 Such a group can be ``split along any normal subgroup'' into a direct product:
 
 \begin{lemma}[Span of Gems - Splitting Lemma I]\label{SpanOfGems}
 If a finite group $G$ is a direct product of simple groups, and $A \lhd G$, then there exists $B\lhd G$, such that $G=A\times B$ as an internal direct product.  
\end{lemma}
\begin{proof} 
We may write $G$ as $S_1 \times \cdots \times S_k$, with $k\geq 1$ and each factor $S_i$ a simple group.
We proceed by induction on $|G/A|$.  If $|G/A|=1$, then we may choose the trivial subgroup of $G$ as $B$ and the assertion holds trivially. 
If $|G/A|>1$, there exists must some $S_i$ with $S_i\not\leq A$ (otherwise $A$ would be all of $G$). Now $S_i$ and $A$ are normal in $G$ properly contained $S_i\cap A$, which is normal in $S_i$, hence trivial. Since the commutator of any $s\in S_i$ and $a\in A$ lies in $S_i\cap A$ it follows that elements of $S_i$ commute with elements of $A$. So $SA=AS \cong S\times A$ is an internal direct and clearly normal in $G$. Now $|G/SA| < |G/A|$, so by induction hypothesis applied to $SA$ we have $G=(A\times S )\times C$, an internal direct product, for some $C\lhd G$. Letting $B=S\times C$, we have $G=A\times B$ as an internal direct product.
\end{proof}

We call $B$ as in Splitting Lemma~I (Lemma~\ref{SpanOfGems}) a {\em complement} to $A$ in $G$.  \\
 
\noindent \underline{Remarks on Splitting Lemma I}:
 The following are not hard to see : 

 1. Note complements are not generally unique as can be seen by considering $\Z_p \times \Z_p$, and $\Z_p \times \1$, which has complements any the other $p^2-p$ subgroups of size $p$.\\ 
 2. The hypothesis of normality of $A$ cannot dropped in if $G$ has a nonabelian simple group as a factor, as, e.g., any proper subgroup of simple group as no complement in $G$.\\
 3. Normality of $A$ can be dropped if $G$ has no simple nonabelian factor. \\
 4. If all the factor $S_i$, then $A$ consists of the factors it intersects has a unique complement consisting of the factors it does not intersect.

The proof of a variant of Lemma~\ref{SpanOfGems} gives more detail on this structure. 
\begin{lemma}[Span of Gems - Splitting Lemma I (Detailed Variant)]\label{detailedSpanOfGems}
Let $G$ be a finite group which is the direct product of simple groups, and let $N$ be a normal subgroup of $G$. Then, 
\begin{enumerate}
\item  \label{normalsubgroup}
$N$ is isomorphic to the direct product of simple groups.
\item  \label{quotient}
 $G/N$ is isomorphic to a direct product of simple groups.
\item  $G \cong N \times G/N$ as an internal direct product.
\end{enumerate}
\end{lemma}
\begin{proof} 
Let $N \lhd G= K_1 \times \cdots \times K_n$ with the $K_i$ simple. 
Let $\pi_i: G\sur K_i$ be the projection onto the $i$th factor.  
Let $I=\{i : \pi_i(N) \neq 1, 1\leq i \leq n\}$.  We may assume 
the first $m \leq n$ factors have $\pi_i(N)\neq \1$. 
Then $N=K_1 \times \cdots \times K_m\times \1 \times \cdots \times \1$.
the first $\ell$ factors are SNAGS and the rest are prime order cyclic groups ($0\leq \ell \leq m$).
By the SNAG rigidity applied successively for each SNAG factor we have that
$$N=K_1 \times \cdots K_{\ell} \times \pi(K_{\ell+1} \times \cdots \times K_m \times \1 \times \cdots \times \1),$$
where   $\pi: G \rightarrow K_{\ell+1} \times \cdots \times K_n$ is projection morphism and
$$\pi(N) \lhd  K_{\ell+1} \times \cdots \times K_m \times \1 \times \cdots \times \1.$$
Now $ K_{\ell+1} \times \cdots \times K_m $ is a product of elementary abelian groups (direct products of prime power order groups), so we may assume $$ K_{\ell+1} \times \cdots \times K_m = (\Z_{p_1})^{n_1} \times \cdots \times (\Z_{p_k})^{n_k},$$
where the $p_i$ are pairwise distinct primes and $n_i>0$. 
It follows from the structure theory of finitely generated abelian groups that
$$\pi(N)\cong (\Z_{p_1})^{m_1} \times \cdots \times (\Z_{p_k})^{m_k},$$
with $0\leq m_i \leq n_i$.  Thus, 
$$N = K_1 \times \cdots \times K_{\ell} \times (\Z_{p_1})^{m_1} \times \cdots \times (\Z_{p_k})^{m_k} \times \1 \times \cdots \times \1.$$
This proves (1). 
Each elementary abelian group  $(\Z_{p_i})^{n_i} $ is a $n_i$-dimensional $\Z_p$ vector space 
and is an internal direct product of its $m_i$ dimensional subspace $(\Z_{p_i})^{m_i}$ and a (non-unique) $n_i-m_i$ dimensional complement $(\Z_{p_i})^{n_i-m_i} $.
It follows that $$G/N\cong \1 \times \dots \times \1 \times (\Z_{p_1})^{n_1-m_1} \times \cdots \times (\Z_{p_j})^{n_k-m_k} \times  K_{m+1} \times \cdots \times K_n,$$
where the first $\ell$ factors are $\1$ and the last $n-m$ are the factors for which $\pi_i(N)=\1$.  This shows (2), and now (3) that 
$$G=K_1 \times \cdots \times K_{\ell} \times (\Z_{p_1})^{n_1} \times \cdots \times (\Z_{p_k})^{n_k} \times K_m \times \cdots \times K_n$$  is the internal direct product of $N$ and $G/N$ follows.
\end{proof}

Remark:  Lemma~\ref{detailedSpanOfGems}(\ref{normalsubgroup})  is not true if the word ``normal'' is removed, since by Lemma~\ref{EmbedInAlternatingGroup} any finite group is a subgroup of a (simple) alternating group.\\[1ex]
 
\subsection{Characterization of the Maximal Complexity Function}\label{sec-mu-complexity}
For each finite group $G$, define the {\em hierarchical complexity} $\mu(G)$ to be the shortest length $\ell$ for which $G$ has a subnormal series 
\begin{equation*}
   \1 =V_0 \lhd V_1 \lhd \cdots \lhd V_{\ell} =G \tag{$\star$}
\end{equation*}
such that $V_{i+1}/V_i$ is isomorphic to a span of gems (i.e.,  the direct product of simple groups) for $0\leq i < \ell$.
We say a series of length $\mu(G)$ of this kind gives 
{\em decomposition of $G$ at its hierarchical complexity}.

We show $\mu$ is bounded by $\JH$ and satisfies the complexity axioms.
\begin{proposition}\label{mu-axioms}
Let $G$ be a finite group. 
\begin{enumerate}
    \item  \label{bounded} $\mu(G) \leq \JH(G) $
     \item (Initial Condition). $\mu(\1)=0$, \label{mu-init}
     \item $\mu(K) =1$ if $K$ is simple or the product of one or more simple groups. \label{mu-on-mala}
     \item  (Constructability Axiom). \label{mu-construct}
     $G$ can be constructed by iterated extension from groups on which $\mu$ is no more than 1.
    \item (Product Axiom).   \label{mu-product} If $G= H \times K$, then $\mu(G)=\max(\mu(H),\mu(K))$.
    
   For $N$ normal in $G$, 
    
   \item       (Quotient Axiom). $ \mu(G/N) \leq \mu(G)$ \label{mu-quotient}
    \item (Extension Axiom). $\mu(G) \leq \mu(N)+\mu(G/N).$ \label{mu-extension}
   
    \item (Normality Property). \label{mu-normal}
    For $N$ normal in $G$, $\mu(N)\leq \mu(G)$.
\end{enumerate}
\end{proposition}
\begin{proof}
(\ref{bounded}) holds since a composition series for $G$ has simple groups has as its quotients the Jordan-Hölder factors of $G$, each of which is trivially a 1-fold product of simple groups. 
(\ref{mu-init}) is trivial, while  (\ref{mu-on-mala}) is immediate from the definition of $\mu$.  

(\ref{mu-construct}) The values of $\mu$ on $\1$, simple groups and their products follow immediately from the definition of $\mu$. This entails constructability by iterated extension from groups on which $\mu$ takes value 1 (cf.\ Lemma~\ref{construct-simple}).

(\ref{mu-product})  Suppose $\mu(H)=a$ and $\mu(K)=b$.
Choose minimal length subnormal series for $H$ and $K$ such that each has quotients that are  spans of gems, i.e., simple or the direct products of simple groups 
$$\1 =H_0 \lhd H_1 \lhd \cdots \lhd H_{a} =H,$$
$$\1 =K_0 \lhd K_1 \lhd \cdots \lhd K_{b} =K,$$
We may assume $a \leq b$.  If $a < b$, define $H_{a+1}= \ldots = H_{b}=H$.     
Now consider the series:
$$\1 =H_0 \times K_0 \lhd H_1 \times K_1 \lhd \cdots \lhd H_{b} \times K_{b}= H \times K.$$
We have $$(H_{i+1}\times K_{i+1}) / (H_i \times K_i) \cong 
H_{i+1} / H_i \times K_{i+1} / K_i .$$  By choice of the series for $H$ and $K$ it follows that these factors are products of finite simple groups.   This shows $\mu(H\times K) \leq \max\{ \mu(H), \mu(K)\}$.
On the other hand, if $G=H\times K$ had a shorter series with direct products of simple groups as its quotients $G_{i+1}/G_i$,
$$\1=G_0 \lhd \cdots \lhd G_\ell=H\times K,$$ then
we could project on the first coordinate with $\pi_H: H\times K \sur H$ to obtain a subnormal series with groups $\pi_H(G_i) \leq H$: Since $\pi_H(gG_i)=\pi_H(g)\pi_H(G_i)$ for all $g\in G_{i+1}$, the projection $\pi_H$ induces a morphism from $G_{i+1}/G_i$ onto $\pi_H(G_{i+1})/\pi_H(G_i)$. By Lemma~\ref{SpanOfGems} (Splitting Lemma I), since $G_{i+1}/G_i$
is a direct product simple groups so is $\pi_H(G_{i+1})/\pi_H(G_i)$.
Now $\pi(G_{i+1})/\pi_H(G_i)$ are the quotients of a subnormal series for $H$ of length less than $\max\{\mu(H),\mu(K)\}$. Similarly, for  $K$. This would contradict the minimality of the chosen of series for $H$ or $K$.  Therefore,
$\mu(H\times K) = \max\{ \mu(K), \mu(K)\}$.

(\ref{mu-quotient})
On the other hand,  applying the natural morphism from $G$ to $G/N$ to the series for $G$ yields a series for $G/N$ in which the successive quotients are homomorphic images of the quotients from the original series \cite[pp. 102-3, first part of proof of Theorem 5.16]{Rotman}. Since homomorphic images of direct products of simple groups are also direct products of simple groups by Lemma~\ref{SpanOfGems}, this series shows $G/N$ has $\mu(G/N) \leq \mu(G)$.

(\ref{mu-extension}) To establish the extension axiom we use strong induction on $\mu(G)$. Let $N\lhd G$. If $\mu(G)=0$, then $G=\1$ and the result is trivial. If  $\mu(G)=1$, then $G$ is the direct product of simple groups. By Lemma~\ref{detailedSpanOfGems}~(Splitting Lemma~I - Detailed Variant), 
$N$ and $G/N$ are each also direct products of simple groups with $G\cong N \times G/N$. So $\mu(N)$ and $\mu(G/N)$ are less than or equal to 1. By (\ref{mu-product}) the direct product axiom for $\mu(G)=\max\{\mu(N),\mu(G/N)\}=1$, and also $\mu(G) \leq \mu(N)+\mu(G/N)$ holds.  
Next suppose the extension property holds for all groups on which  with $\mu$ is less than $\ell$.   Let $G$ be such that $\mu(G)=\ell$,   and take a subnormal minimal series of $G$ whose quotients are products of simple groups. 

Let $N \lhd G$ be a proper normal subgroup.   We then have a series for $N$ obtained from the series for $G$ of length $\ell$ 
as in ($\star$) above.

$$\1 =V_0 \cap N \lhd V_1 \cap N \lhd \cdots \lhd V_{\ell}\cap N =N,$$

Then
each factor is 
$$W_{i}= (V_{i+1}\cap N) / (V_i \cap N).$$ Now by  Noether's 
isomorphism theorem  that $HK/K\cong H/(H\cap K)$ for $K$ normalized by $H$, we can take $H=V_{i+1}\cap N$ and $K=V_i$ (which is normal in $V_{i+1}$ hence normalized by $V_{i+1}\cap N$), to conclude that 
$$(V_{i+1}\cap N)V_i/V_i \cong (V_{i+1}\cap N)/(V_i\cap (V_{i+1}\cap N)) = (V_{i+1}\cap N)/(V_i \cap N)=W_{i}.$$
But note that 
$$(V_{i+1}\cap N)V_i/V_i \,\, \leq \,\, V_{i+1}/V_i,$$
since 
$V_{i+1}\cap N$ and $V_i$ are subgroups of $V_{i+1}$. 
Moreover,  since $(V_{i+1}\cap N)V_i$  is normal in $V_{i+1}$, 
$$(V_{i+1}\cap N)V_i/V_i \,\, \lhd \,\,  V_{i+1}/V_i,$$ by
the Correspondence Theorem (e.g.\cite[Theorem~2.28]{Rotman}).
By choice of the series for $G$,  $V_{i+1}/V_i$ is the direct product of simple groups, so it follows by Lemma~\ref{detailedSpanOfGems} (Splitting Lemma I)
that $(V_{i+1}\cap N)V_i/V_i\cong W_i$ is the direct product of simple groups.   
Thus we have obtained series for $N$ witnessing the fact that $\mu(N) \leq \ell =\mu(G)$.

Here is an alternative proof of
(\ref{mu-extension}). Now we show $\mu(G) \leq \mu(N) + \mu(G/N)$:
Taking minimal series of whose whose groups are direct products of simple groups for $N$ and $G/N$

$$\1 =A_0  \lhd A_1   \lhd \cdots \lhd A_{a}  = N,$$
$$\1 =B_0  \lhd B_1  \lhd \cdots \lhd B_{b} =G/N,$$
and letting $\varphi:G\sur G/N$ be the natural quotient map, to obtain
$$N = \varphi^{-1}(B_0) \lhd \varphi^{-1}(B_1) \lhd \cdots \lhd \varphi^{-1}(B_b) = G,$$
with quotients $\varphi^{-1}(B_{i+1})/\varphi^{-1}(B_i)$ isomorphic to  $B_{i+1}/B_i$ by the Correspondence Theorem (e.g.\ \cite[Theorem~2.28]{Rotman}).  Concatenating the first and third series above yields a subnormal series for $G$,  with all quotients direct products of simple groups, establishing $\mu(G) \leq  a + b = \mu(N)+\mu(G/N)$. \end{proof}

\begin{theorem}\label{mucpx}
The hierarchical complexity measure $\mu$ satisfies the complexity axioms and dominates any other complexity function $\c$.   Hence, $\mu$ is $\cx$, the unique maximal complexity measure on groups.
\end{theorem}
\begin{proof} 
By Proposition~\ref{mu-axioms}, $\mu$ satisfies the complexity axioms. Let $\c$ be any complexity measure satisfying the axioms. 
Then $\mu(\1)=\c(\1)=0$.   By definition of $\mu$,  if $\mu(G)=1$ then $G$ is a non-trivial direct product of simple groups.  For simple groups $K$, $\mu(K)=1 \geq \c(\1)$ by Lemma~\ref{simplecpx} for $\c$.  Thus,  $\c(G)$ is the $\max$ of $\c$ on the simple factors of $G$ by the product axiom, hence at most one, i.e, $\mu(G)=1\geq \c(G)$.

We now continue by induction on $\mu(G)=\ell > 1$:

Take a shortest series for $G$ with quotients isomorphic to direct products of simple groups :

$$\1 =G_0 \lhd G_1 \lhd \cdots \lhd G_{\ell-1} \lhd G_{\ell}=G =G,$$

This series must a shortest one for $G_{\ell-1}$ of this type, lest $G$ had a shorter such series. Therefore $$\mu(G_{\ell-1})=\ell-1$$
and by induction hypothesis
$\c(G_{\ell-1})\leq \mu(G_{\ell-1})$, and so by the extension axiom for $\c$
$$\c(G) \leq \c(G_{\ell-1})+\c(G/G_{\ell-1}) \leq (\ell-1)+1 =\mu(G).$$ 

We conclude by induction that $\mu$ dominates an arbitrary complexity function $\c$. Thus, by  Theorem~\ref{uniqueCpx}, $\mu=\cx$
\end{proof}
 
\begin{corollary} A nontrivial group has hierarchical complexity 1 if and only if it is a direct product of simple groups. That is, for a finite group $G\neq \1$, $$G \in \Span(\SIMPLE) \, \,  \Leftrightarrow \, \, \cx(G)=1.$$
\end{corollary}

\noindent We observe that $\cx$ satisfies the normal property though it was not assumed: 
\begin{theorem}[Normal Property of Hierarchical Complexity]\label{cpxNorm}
If $N$ is a normal subgroup of $G$, then $\mu(N)\leq \mu(G)$. Hence, $\cx(N)\leq \cx(G)$. 
\end{theorem}
\begin{proof}
If $N$ is normal in $G$, a minimal subnormal series for $N$ with quotients all spans of gems can be extended to one for $G$. Thus $\mu(N) \leq \mu(G)$.  By Theorem~\ref{mucpx}, $\mu=\cx$. 
\end{proof}

Do the complexity axioms, which do not include the normal property, imply it for an arbitrary $\c$ satisfying them? 
No, the normal subgroup property is independent of the other axioms (as shown in Section~\ref{normal}). However, it holds  for $\mu=\cx$ and many other important complexity functions~(see Sections~\ref{solv} and \ref{der-fit-solv-cpx}).

\subsection{Computing Hierarchical Complexity}

Let $G$ be a finite group. By Theorem~\ref{mucpx},  $\cx(G)=\mu(G)$ can be computed by examining all subnormal chains whose quotients are spans of gems to find one of minimal length.  
To make determination of $\cx(G)$ easier,  we can often compute or bound complexity of groups already from the axioms. 
Moreover, the normal property, as well as upper and lower bounds on $\cx$ can simplify this further (Theorems~\ref{cpxNorm}, \ref{lowerbounds}   \& \ref{upperbounds}, and other results).

\begin{example}[Span of Gems]
By definition of $\mu$,  if $\mu(G)=1$ then $G\neq \1$ is  a direct product  $G_1 \times \cdots\times G_n$  of finite simple groups $G_i$ ($1\leq i \leq  n$, $n\geq 1)$. 
So by Theorem~\ref{mucpx}, $\cx(G)=1$ if and only if $G$ is of this form, i.e., $G$ is a ``span of gems". 

\end{example}

\begin{example}[Two Jordan-H\"older Factors]\label{2JH}
 Let $G$ be an extension of $Q$ by $N$ with $Q$ and $N$ simple. We have $\cx(Q)=\cx(N)=1$.   By the extension axiom $\cx(G) \leq \cx(Q) + \cx(N)=2$.  
  If $G$ is the direct product $Q \times N$, $\cx(G)=1$ by the
product axiom.  Otherwise, $\cx(G)=2$ since the unique subnormal series  $\1 \lhd N \lhd G$ has length two, and quotients are spans of gems since they are simple.
\end{example}

\begin{example}[Symmetric Groups]\label{SymExample} Let $S_n$ be the symmetric group of
all permutations on $\{1,\ldots, n\}$, and let $A_n$ be the alternating group, the subgroup  of even permutations. 
If $n=2$, $S_2\cong\Z_2$ is simple and has complexity $1$. 
Now $S_n/A_n \cong \Z_2$ and for $2< n\neq 4$, $A_n$ and $\Z_2$ are simple, so $\cx(S_n))=2$ by example~\ref{2JH}. 
For $n=4$, $\cx(S_4)\leq \cx(A_4)+\cx(\Z_2)$.  $A_4$ has a normal subgroup $\Z_2\times \Z_2 \cong \langle (1,2)(3,4), (1,3),(2,4)\rangle$,
with quotient $\Z_3$, so $\cx(A_4)\leq \cx(\Z_2\times \Z_2) +\cx(\Z_3)=2$ by the extension and product axioms, and the simplicity
of $\Z_2$ and $\Z_3$.  It follows $\cx(S_4)\leq 3$.   
However, inspection shows $1\lhd  \langle \Z_2\times \Z_2\rangle\lhd  A_4 \lhd S_4$ is a unique minimal subnormal series for $S_4$ with quotients all spans of gems, so $\cx(S_4)=3$.
\end{example}
\begin{example} \label{ZpnExample}
Let $\Z_{p^n}$ denote the cyclic group generated by $x$ of order $p^n$ for $p$ prime, $n\geq 1$.
Then $\Z_{p^n}$ is has a normal subgroup $\langle x^p \rangle$ isomorphic to $\Z_{p^{n-1}}$ with quotient $\Z_{p^n}/\Z_{p^{n-1}}\cong \Z_p$,
so by the extension axiom, $\cx(\Z_{p^n})\leq \cx(\Z_{p^{n-1}}) + \cx(\Z_p)$. Since $\cx(\Z_p)=1$, it follows by induction that
$\cx(\Z_{p^n})\leq n$.  (We shall see that $n$ is also a lower bound in example~\ref{Zpwr}.)
\end{example}

\noindent {\em The number of factor groups in any minimal decomposition at complexity is $\mu(G)$ but neither the order nor the isomorphism classes of the quotients need to be the same:}

\begin{counterexample}[Order of Factors Not Unique at Complexity] The group of symmetries of the square, $D_4\cong \Z_2\wr \Z_2$ has
$\mu(D_4)=2$ but the quotients $\Z_2$ and $\Z_2\times \Z_2$ can appear in either order in decompositions at complexity.\\ \label{D4}
\begin{tabular}{|c|c|}
                                         
	\hline
                                         
	Group &  Subnormal Chain \\
                                         
	\hline

	$D_4$ & = $\langle (1,2), (3,4), (1,3)(2,4)\rangle $ \\ 
	$|{}$ & quotient $\Z_2\times \Z_2 $ \\ 
	$\Z_2$ &$=\langle   (1,2)(3,4)\rangle $ \\ 
	$|{}$ & quotient $\Z_2$ \\ 
	$\1$ & = $\1$\\

	\hline
                                         
\end{tabular}
\begin{tabular}{|c|c|}
                                         
	\hline
                                         
	Group & Subnormal Chain \\
                                         
	\hline

	$D_4$ & =$\langle  (1,2), (3,4), (1,3)(2,4)\rangle $ \\ 
	$|{}$ & quotient $\Z_2$\\ 
	$\Z_2\times \Z_2$ & = $\langle (1,2)(3,4), (1,2) \rangle $\\ 
		$|{}$ & quotient $\Z_2 \times \Z_2$ \\ 
	$\1$ &  $=\1$ \\

	\hline
                                         
\end{tabular}
\end{counterexample}

\begin{counterexample}[Isomorphism Types of Factors Not Unique at Complexity]\label{FactorsNotUnique}
Also the isomorphism types of the quotients are not unique for decompositions at complexity,
for $\Z_2\wr\Z_4$. Inspection of all chains of subnormal subgroups shows complexity $\mu(\Z_2\wr\Z_4)=3$, where all three quotients may be isomorphic to $\Z_2\times \Z_2 \times \Z_2$,
or the non-isomorphic $\Z_2$, $\Z_2\times \Z_2$, $\Z_2\times\Z_2\times \Z_2$ (with order not unique) for different subnormal chains, or even   $(\Z_2)^4$, $\Z_2$, $\Z_2$ for
the subnormal series
$$\1\lhd (\Z_2)^4\cong (\Z_2 \times \Z_2 \times \Z_2 \times \Z_2) \rtimes \1 \lhd (\Z_2 \times \Z_2 \times \Z_2 \times \Z_2) \rtimes \Z_2 \lhd \Z_2\wr \Z_4,$$
Either $\Z_2$ or $\Z_2\times \Z_2$ must occur as the topmost quotient of $\Z_2\wr \Z_4$, but with this 
constraint all possible orderings of direct products of the six $\Z_2$ Jordan-H\"older factors of $\Z_2\wr\Z_4$ occur in some decomposition at complexity.  Here are some more examples:

\begin{tabular}{cl}

 $\Z_2\wr\Z_4 = \langle (1,2), (3,4), (5,6), (7,8), (1,3,5,7)(2,4,6,8) \rangle$ \\ 
$|$ & quotient $\Z_2$  \\ 
$ (\Z_2 \times \Z_2 \times \Z_2) \rtimes \Z_4 \cong \langle (5,6)(7,8), (3,4)(5,6), (1,2)(3,4), (1,3,5,7)(2,4,6,8) \rangle$  \\ 
	$|$ & quotient  $\Z_2 \times \Z_2$  \\ 
$\Z_2 \times \Z_2 \times \Z_2  =\langle (1,2)(5,6), (3,4)(7,8), (1,5)(2,6)(3,7)(4,8) \rangle$

\end{tabular}\\[2ex]

Also, the subnormal series \\

\begin{tabular}{cl}

	 $\Z_2\wr\Z_4 = \langle (1,2), (3,4), (5,6), (7,8), (1,3,5,7)(2,4,6,8) \rangle$ \\ 

	$|$ &  quotient $\Z_2 \times \Z_2 $  \\
 $\Z_2 \times D_4 \cong \langle (5,6)(7,8), (3,4)(5,6), (1,2)(3,4), (1,5)(2,6)(3,7)(4,8) \rangle$ \\

\end{tabular}\\[2ex]

 may be continued to a subnormal series at complexity 3 by various normal subgroups of $\Z_2\times D_4$ yielding non-isomorphic quotients in different orders: \\

\begin{tabular}{cl}
	$|$ & quotient $\Z_2 $  \\
 $\Z_2 \times \Z_2 \times \Z_2  \cong \langle (1,2)(3,4)(5,6)(7,8), (1,5)(2,6)(3,7)(4,8), (3,4)(7,8) \rangle$ \\[1ex] 

or \\

	$|$ & quotient $\Z_2 \times  \Z_2$  \\ 
 $\Z_2 \times  \Z_2 \cong \langle  (1,2)(3,4)(5,6)(7,8), (1,5)(2,6)(3,8)(4,7) \rangle$ \\[1ex]
 
 or \\
	$|$ & quotient $\Z_2 \times \Z_2 \times \Z_2$  \\
$ \Z_2 \cong  \langle (1,2)(3,4)(5,6)(7,8) \rangle $ \\

\end{tabular}\\[2ex]

In fact, there are 15 different  subnormal subseries at complexity for $\Z_2\wr\Z_4$.
\end{counterexample}

\begin{remark}[Diversity of Constructions at Complexity] \hfill\\

\vspace{-0.5cm}
\begin{enumerate}
\item  There may be multiple ways to build a group by a {\em minimal number} of extensions from simple groups and their direct products (as Counterexamples~\ref{D4} and \ref{FactorsNotUnique} show).
In other words, `the' manner of constructing a group $G$ by iterated extension even {\em in a minimal way} (i.e., at complexity) is not always unique in terms of type of `pieces' and the order of extension to constitute the `whole'. There can be multiple dissimilar solutions to building $G$ using $\cx(G)$ pieces of complexity~1 in terms of how Jordan-H\"older factors are distributed amongst the `pieces' (i.e., constituents of the iterated extension). 
\item Iterated extensions using isomorphic pieces in the same order may correspond to essentially different minimal constructions of $G$. For example, in  the first 3-m\={a}l\={a} minimal decomposition of $\Z_2\wr\Z_4$ in Counterexample~\ref{FactorsNotUnique}, the normal group of $ (\Z_2 \times \Z_2 \times \Z_2 \times \Z_2) \rtimes \Z_2 =\langle (1,5)(2,6)(3,7)(4,8), (7,8), (5,6), (3,4), (1,2) \rangle$ can be replaced by $\langle (3,4)(7,8), (1,2)(5,6), (5,6)(7,8) \rangle\cong \Z_2\times \Z_2 \times \Z_2$, to yield a minimal construction of $\Z_2\wr\Z_4$ at complexity
with the same constituents $\Z_2$, $(\Z_2)^2$,  $(\Z_2)^3$ in the same order as  in the second decomposition in Counterexample~\ref{FactorsNotUnique} above. Although built from the same pieces $(\Z_2)^2$ and  $(\Z_2)^3$ in the same order, the middle groups are of the form $(\Z_2 \times \Z_2 \times \Z_2) \rtimes \Z_4$ and $(\Z_2 \times \Z_2 \times \Z_2 \times \Z_2) \rtimes \Z_2$ but not isomorphic, since as can be shown they have centers of size 2 and 4, respectively. Nevertheless, $\Z_2$ can be extended by either of these middle groups to yield a minimal construction of $\Z_2\wr\Z_4$. 
\end{enumerate}
\end{remark}

\begin{example}[Nilpotent Groups]\label{nilpotentcpx}
Since a nilpotent group $N$ is a direct product $N=P_1\times \cdots \times P_k$ of its Sylow $p$-groups (e.g., \cite[Theorem~10.3.4]{Hall}),   by the product axiom, $\cx(N)$ is the maximum complexity $\cx(P_i)$.  This reduces the computation of the complexity of nilpotent groups to computing complexity of $p$-groups. 
\end{example}

\begin{example}[Solvable Group Complexity Bounds]
If $G$ is a solvable group, $G$ has a minimal length subnormal series in which quotients $N_i$ are nilpotent. The length of such a series is 
the Fitting height of $G$  (see also Section~\ref{der-fit-solv-cpx}, Rhodes' Lemma~\ref{RhodesLemma}). 
Refining this series using  minimal length series for the $N_i$ whose quotients are spans of gems from example~\ref{nilpotentcpx} yields a subnormal series whose length is an upper bound for $\cx(G)$.
By the extension property  of complexity,
$$\cx(G)\leq \sum_{i=1}^n \cx(N_i)  \leq \sum_{i=1}^n\sum_{j=1}^{k_j} \max\{\cx(P_1^{(i)}), \ldots , \cx(P_{k_j}^{(i)})\},$$
where $P_1^{(i)}, .., P_{k_j}^{(i)}$ ($k_j\geq 1$) denote the Sylow $p$-subgroups of the nilpotent group $N_i$ as in example~\ref{nilpotentcpx}. 
\end{example}

\subsection{Wreath Product of Complexity 1 Groups}

We can apply some of what we have just seen to compute the (hierarchical) complexity of the wreath product of two complexity 1 groups.

\begin{proposition}\label{wreathcpx1}
If $N$ and $Q$ have hierarchical complexity 1, then the wreath product $N\wr Q$ has hierachical complexity 2. 
In particular,  if $N$ and $Q$ are simple groups, their wreath product  has complexity 2. 
\end{proposition}
\begin{proof}
The wreath product $W=N\wr Q$ is an extension of $Q$ by the $|Q|$-fold direct product $N^Q= N \times \cdots \times N$,
and is a semidirect product $N^Q\rtimes Q$.
By the extension and product axioms $\cx(W)\leq \cx(Q) + \cx(N^Q) = \cx(Q)+\cx(N)$.  

If $N$ and $Q$ are spans of gems (complexity 1),  then $0\neq \cx(W)\leq 2$. We show $\cx(W)=2$. 
Suppose for a contradiction that $\cx(W)=1$ then we would have $N\wr Q $ a span of gems.
Consider the natural morphism $\pi: W \sur Q$. The kernel of $\pi$ is $N^Q\rtimes\1 \lhd W$ and $\ker(\pi)\cong N^Q$. 
By Splitting Lemma~I (Lemma~\ref{SpanOfGems}),  $W$ is an internal direct product $(N^Q\rtimes\1) \times B$, for $B$ a normal complement to $(N^Q\rtimes\1)$ in $W$, 
$B\lhd W$ with $B\cap (N^Q\rtimes\1) $ is trivial.  We have $B\cong W/(N^Q\rtimes \1)\cong Q$.   
Consider the subgroup $\1^Q\times Q$ of $W$.  Choose any $(c_1,q)\in \1^Q\times Q$ where $c_1: Q\rightarrow \1 \leq N$ is the constant function, with $q\neq 1$.
Since $W=(N^Q\times \1)B$, we may write $(c_1,q)$ as 
$$(c_1,q)=(w_2,1)(b_2,b_1)\mbox{ with }(w_2,1)\in N^Q\times \1 \mbox{ and }(b_2,b_1)\in B$$
Here $w_2, b_2 \in  N^Q N$, and $1, b_1 \in Q$.   Now $(c_1,q)=(w_2,1)(b_2,b_1)=(w_2b_2,b_1)$.  Hence $b_1=q$ and $c_1=w_2b_2$,
so $b_2=w_2^{-1}$ 
Now since $B\lhd W$, we have 
$(c_1,q)^{-1} (b_2,b_1) (c_1, q)= (c_1, q^{-1})(w_2^{-1},q)(c_1,q)= (w_2^{-1}, q^{-1}qq)=(w_2^{1}, 1) \in B$.
So,  $(w_2^{-1},1)\in (N^Q\times 1)\cap B=\1$. This shows $(b_2,b_1)$ is conjugate to the identity, hence $(b_2,b_1)=1$.
Hence $(c_1,q)=(w_2,1)(b_2,b_1)=(w_2,1)$.  This contradicts $q\neq 1$. Hence $\cx(W)\neq 1$, but rather $\cx(W)=2$.
\end{proof}

\noindent\underline{Remark}: One can easily check $\1^Q\rtimes N$ is not normal in $W=N\wr Q$, but this is implied by the above
argument, for otherwise $\1^Q\rtimes Q$ would be a normal complement to $(N^Q\rtimes\1)$ in $W$.

\section{Minimal Normal Subgroups, Socle Length and Hierarchical Complexity}\label{MinSocHier}
 To begin to develop the `semi-local' theory relating finite groups of complexity 1 (spans of gems) occurring as hierarchical layers within the structure of higher complexity groups that contain them, this section develops material related on minimal normal subgroups (e.g., \cite{Hall}, \cite{DixonMortimer}, \cite{Rotman})  and the socle of a finite group (e.g.,  \cite{DixonMortimer}, \cite{CannonHolt}).
While much of this is classical and well-known,  there are also some new formulations and results that we shall need subsequently for the complexity theory of finite groups. These include more detailed splitting and fragmentation lemmas of Sections~\ref{1stSplittingLemmas}~and~\ref{FragSplitII},  and lead to proofs of original results that socle length satisfies the normal and quotient axioms, but {\em not} the extension axiom, and is a sharp upper bound for hierarchical complexity.

 \subsection{Minimal Normal Subgroups}
 A normal subgroup $M\neq \1$ of a group $G$ is a {\em minimal normal subgroup of $G$} if  it contains no other non-trivial normal subgroup of $G$. That is,  $\1 \neq K \lhd G$  with  $K \leq M$ implies $K=M$.
 We make some observations on minimal normal subgroups.  By the Correspondence Theorem (e.g., \cite[Theorem 2.28]{Rotman}), the following is clear.

 \begin{lemma}\label{Min}
The image of a minimal normal subgroup under a surjective map is either a minimal normal subgroup of the image or is trivial.
\end{lemma}

 We shall make use of the following well-known elementary Fact. 
 
  \begin{fact}\label{elemLemma}
   If $A$ and $B$ are normal subgroups of a group $G$ and $A\cap B=\1$, then $ab=ba$ for all $a\in A$, $b\in B$. Moreover, the group $\langle A, B\rangle$ generated by $A$ and $B$ equals $AB \lhd G$, and $AB$ is the internal direct product isomorphic to $A \times B$ \end{fact}
   
  \begin{proof} $(b^{-1}a^{-1}b)a \in A$ since $b^{-1}a^{-1}b\in A$, and similarly $b^{-1}(a^{-1}ba)\in B$ since $a^{-1}ba \in B$. So $b^{-1}a^{-1}ba=1$, i.e., $a$ and $b$ commute. In any product in $G=\langle A, B \rangle$, we can move all $a$'s to the left and all $b$'s to the right, so $\langle A, B\rangle=AB$. To multiply in $AB$, we
  have $(ab)(a'b')=a(ba')b'=aa'bb'$ ($a,a'\in A, b, b'\in B$).   Also for $g\in G$, $ab\in AB$, we have  $g^{-1}abg=(g^{-1}ag)(g^{-1}bg)\in AB$, so $AB$ is normal in $G$.
  Each element of $AB$ can be uniquely written in
  the form $ab$ : for if $ab=a'b'$  then it follows $a^{-1}a'=bb'^{-1}\in A\cap B=\1$, whence $a=a'$ and $b=b'$.
  Thus the map  $(a,b)\mapsto ab$ is a surjective and injective morphism, showing  $A\times B\cong AB$. 
 \end{proof}
 
 \begin{fact}\label{ProdMinNorm}
\begin{enumerate}
\item
Let  $N_1, \ldots, N_n$ be minimal normal subgroups of $M$. Then the subgroup they generate is an internal direct product and normal in $M$.  Furthermore,  
$\langle N_1, \ldots, N_n \rangle \cong N_{i_1} \times \ldots \times N_{i_m}, $
for some $1 \leq m \leq n$ and $1 = i_1 < \cdots < i_m \leq n$.
    \item Let  $M$ be a minimal normal subgroup of $G$, and $N$ be a minimal normal subgroup of $M$. 
    Then $M$ is the internal direct product of $N$ and (zero or more) conjugates of $N$. 
\end{enumerate}
\begin{proof}
 (1)  Inductively define $M_1=N_1$, and for $0 \leq i<n$, 
 $$M_{i+1}:= \begin{cases}
  \langle M_i, N_{i+1} \rangle  &  \mbox{ if } \1 = M_i \cap N_{i+1} \\
      M_i     & \mbox{ otherwise.}
 \end{cases}$$
 We proceed by induction from 1 up to $n$, to show that $M_i = \langle N_j : 1 \leq j \leq i\rangle$, that  $M_i$ is the internal direct product of a subset of $\{N_1, \ldots , N_i\}$,  and that $M_i \lhd M$:   This is trivially for $i=1$, since $M_1=N_1 \lhd M$.  Now suppose this assertion is true for $i$ ($1 \leq i<n$).  
 On the one hand, 
if $\1 = M_i \cap N_i$,  by Fact~\ref{elemLemma}, $M_{i+1}=\langle M_i, N_{i+1} \rangle= M_i N_{i+1} \cong M_i \times N_{i+1}$, and $M_{i+1} \lhd M$. By induction hypothesis $M_i=\langle N_j : 1 \leq j \leq n \rangle$ is the internal direct product a subset of   $\{N_1, \ldots , N_i\}$, so it follows $M_{i+1}$ is the internal  direct product of a subset of 
$\{N_1, \ldots N_{i+1}\}$.    On the other hand if $\1 \neq M_i \cap N_{i+1}$, then  $ M_i \cap N_{i+1} \lhd M$ since $M_i \lhd M$ (by induction hypothesis) and $N_{i+1}\lhd M$. So minimal normality of $N_{i+1}$ in $M$ entails $M_i \cap N_{i+1} =N_{i+1}$, so $N_{i+1} \leq M_i$.  In either case, $N_{i+1} \leq M_{i+1}$, and so $M_{i+1}=\langle M_i, N_{i+1} \rangle = \langle N_1, \ldots, N_i, N_{i+1}\rangle.$ This proves the assertion for $i+1$.  By induction, the assertion holds $i=n$ and (1) is proved.
 
 (2) Let $g_1(=1), \ldots, g_n$ be elements of $G$, such that the conjugates $N_i=g_i^{-1} N g_i$ are pairwise distinct for $i\neq j$ ($1\leq i, j \leq n$), and each conjugate of $g^{-1}Ng$ in $G$ occurs as  $g_i^{-1}Ng_i$ for some $i$.
  Since  conjugation is an automorphism of $G$, 
we have $g^{-1} N g $ is minimal normal in $g^{-1} M g=M$ (since $M \lhd G$).
Applying (1) yields that
$$M': =\langle g^{-1} N g : g\in G \rangle = \langle g_i^{-1} N g_i : 1 \leq i \leq n \rangle = 
 N_{i_1} \ldots N_{i_n} \lhd M, \mbox{ and } $$
 $$M'\cong  N_{i_1} \times  \ldots \times N_{i_n},$$ where $i_1=1$.
Clearly, $M'$ is the normal closure of $N$ in $G$, hence by minimal normality of $M$ in $G$, $M'=M$.  
\end{proof}
 \end{fact}

 \begin{fact}\label{MinNormSubgroup}
 A minimal normal subgroup $M$ of a finite group $G$ is simple or the (internal) direct product of isomorphic simple groups.
 \end{fact}
 \begin{proof}  Let $S$ be a minimal normal subgroup of $M$.   
 By Fact~\ref{ProdMinNorm}(2), $M$ is the internal direct product of conjugates of $S$ (including $S$ itself). 
 Suppose $\1 < K  \lhd S$, then $S$ normalizes $K$ and  each internal direct product factor different from $S$  commutes with $S$, hence with $K$.   Thus $K \lhd M$.   Mimimal normality of $S$  in $M$ now implies $K=S$. Therefore $S$ is simple.
 \end{proof}
 
 In particular, a minimal normal subgroup is a span of gems.

\subsection{Fragmentation and Splitting Lemma II}\label{FragSplitII}
 
The {\em socle} $\soc(G)$ of a finite group is the subgroup of $G$ generated by the minimal normal subgroups of $G$.  

\begin{theorem} For a finite group $G\neq \1$, $\cx(\soc(G))=1$.  \end{theorem}
\begin{proof}
It is the internal direct product of minimal normal subgroups  and normal in $G$ by Lemma~\ref{ProdMinNorm}(1). By Fact~\ref{MinNormSubgroup}, each
of these minimal normal subgroups is a direct product of simple groups, so $\soc(G)$ is too. It follows $\cx(G)=\mu(G)=1$.
\end{proof}

If $\soc(G)$ intersects the normal subgroup $N$ of $G$ non-trivially,  a minimal normal subgroup $M$ of $G$  completely ``fragments'' as a product of isomorphic minimal normal subgroups of $N$, and lies in $\soc(N)$:

\begin{fact}[Fragmentation Lemma]\label{fragment}

Let $M$ be a minimal normal subgroup of $G$ and $N \lhd G$. Suppose $M \cap N\neq \1$. Then 
\begin{enumerate} 
\item $M$ is internal direct product of isomorphic minimal normal subgroups of $N$.  \label{frag}
\item  $M=M\cap N$.
\item  (Absorption 1). $M\lhd N$.
\item (Absorption 2).  $M \lhd \mbox{\rm soc}(N).$  
\end{enumerate}
\end{fact}
\begin{proof} (2) Since $N\lhd G$, $\1\neq M\cap N$ is normal in $G$. Since $M$ is minimal normal in $G$, $M\cap N=M$.
(3) follows from (2), since $M\cap N \lhd N$.
(1) We have $\1\neq M\cap N \lhd N$. Let $K\neq \1$ be a minimal normal a subgroup of $N$ contained in $M\cap N$. 
For any $g\in G$,  $g^{-1}Kg$ lies in $N$ since $N \lhd G$. Since conjugation is an automorphism, $g^{-1}Kg$ is a minimal normal subgroup of $N=g^{-1}Ng$.  
By  minimal normality  in $N$, either $K\cap g^{-1}Kg=1$ or $K=g^{-1}Kg$. Choose
representatives
$g_1,\ldots g_k \in G$ so that $g_1^{-1}Kg_1, \ldots, g_k^{-1}Kg_k$ comprise the collection of all the distinct conjugates of $K$. 
Since $K\leq M\cap N= M$ (by (2)) and each $g^{-1}Kg \leq g^{-1}Mg=M$, therefore, $M'=\langle g^{-1} K g : g \in G\rangle \leq M$, but this is the normal closure of $K$ in $G$ so $M'\lhd M$. By  minimal normality  of $M$ in $G$, $M'=M$.
By Fact~\ref{ProdMinNorm}(1),  the group $M'$ is normal in $N$ and is isomorphic to an internal direct product copies of $K$. 
Therefore $M$ is the internal direct product of $K$ and conjugates of $K$ in $G$.
(4) By (1), $M$ is the product of minimal normal subgroups of $N$, hence $M \lhd \soc(N)$.
\end{proof}

\begin{lemma}[Span of Gems - Splitting Lemma II]  \label{SpanOfGems2}
Let $N$ be a normal subgroup of a finite group $G$, then
\begin{enumerate}
\item $\soc(G) \cap N =M_1 \ldots  M_k = N_1 \ldots N_m,$
is the internal direct product of some minimal normal subgroups $M_i$ of $G$ ($m\geq 0$, $1 \leq i \leq m$),  
and also the internal direct product of some minimal normal subgroups $N_j$ of $N$ ($m \geq 0$, $1 \leq j \leq m \leq n$). 
\item In particular, $\soc(G) \cap N \leq \soc(N)$, and $\soc(G)\cap N \lhd \soc(N)$. 
\item The complement of $N_1\ldots N_m$ in $\soc(G)$ is an internal direct product of minimal normal subgroups $M'$ of $G$ that intersect $N$ trivially. That is, for some  $M_1', ..., M_s'$ minimal normal in $G$ 
  $$\soc(G) = (\soc(G) \cap N) \times M_1' \times  \ldots \times  M'_s,$$ is an internal direct product,
  with each $M_i'\cap \soc(N)=\1$, $s\geq 0$ ($1\leq i \leq s$).
\item The complement of $N_1 \ldots N_M$ in $\soc(N)$  is the internal direct product of minimal normal subgroups of $N'$ that intersect $\soc(G)$ trivially. That is, for some $N_1', ..., N_t'$ minimal normal in $N$ 
$$\soc(N) = (\soc(G) \cap N) \times N_1' \times \ldots \times  N'_t,$$ 
is an internal direct product, with each $N_i'\cap \soc(G)=\1$,  $t\geq 0$ ($1\leq i \leq t)$. 
\end{enumerate}
\end{lemma}
\begin{proof}
(1) Let $ P=\soc(G) \cap N$.  Then $P \lhd G$, since $\soc(N)$ and $P$ are normal in $G$, and so $P \lhd \soc(G)$. 
By Lemma~\ref{detailedSpanOfGems}, $P$ is a direct product of (zero or more) simple groups since $P\lhd \soc(G)$. 
Hence $P=\langle S : S \mbox{ simple, }S\lhd P \rangle$.
Note for simple $S\lhd P$, we have $P=S \times Q$  and $\soc(G)=P \times B$ by Lemma~\ref{detailedSpanOfGems}, so $\soc(G)=S\times Q\times B$ and $S\lhd \soc(G)$.
Let $M_S=\langle g^{-1} S g : g \in G\rangle$ which is a a minimal normal subgroup of $G$. 
Every conjugate $g^{-1}Sg$ of $S$ is simple and normal in $g^{-1}Pg=P$.   So $M_S \leq P$.  Thus 
$$P = \langle S :  S \mbox{ simple, }S\lhd P \rangle \leq \langle M_S :  S \mbox{ simple, }S\lhd P \rangle \leq P.$$
So by Fact~\ref{ProdMinNorm}(1) applied to the set of these minimal normal subgroups $M_S$ of $G$, 
we have that $P= \langle M_S :  S \mbox{ simple, }S\lhd P \rangle=M_{S_1}\ldots M_{S_n}$ is the internal direct product for some simple $S_1, \ldots S_n \lhd P$.
 Then, since,  $\1\neq S \leq M_S \cap N$, 
the Fragmentation Lemma (Lemma~\ref{fragment}\ref{frag}) shows that each $M_S$ is the internal direct product of some minimal normal subgroups of $N$.
We have that $P$ is the internal direct product of some minimal normal subgroups of $N$.
(2)~Since $\soc(N)$ is the generated by all the minimal normal subgroups of $N$, this follows from (1). Also $\soc(G) \cap N=P$ is normal in $G$, hence in $N$.
(3) follows from Lemma~\ref{ProdMinNorm}(1) applied  the set consisting of
$M_1, \ldots, M_n$ and all the minimal normal subgroups of $G$ that intersect $N$ trivially, since these together generate $\soc(G)$. 
Similarly, (4) follows from  Lemma~\ref{ProdMinNorm}(1) applied to the set consisting of $N_1, \ldots N_m$ and all $N'$ minimal normal in $N$ that intersect $\soc(G)$ trivially, since these together generate  $\soc(N)$. 
\end{proof}

\subsection{A Characteristic Series and Complexity Upper Bound}
\label{socleSeries}

\begin{definition}[Socle Characteristic Series and $\sx$] \label{length} Let $G$ be a finite group.  Let $V_0=\1$,   
define  $V_1=\soc(G)$,  and inductively for $i \geq 1$, define the natural surjective homomorphism  $h_i: G\sur G/V_i$,  $\soc_{i+1}(G)=\soc(G/V_i)$, $V_{i+1}=h_i^{-1}(\soc_i(G))$.  Since $\soc(G/V_i)$ is normal in $G/V_i$, its inverse image $V_{i+1}$ is normal in $G$.
Observe that if $V_i\neq G$, then  $V_{i+1}$ properly contains $V_i$.  
Therefore, $V_{\ell}=G$ for some $\ell \geq 0$. Let $\sx(G)$, {\it the socle length} (or {\it socle `complexity'}) of $G$, be the least integer $\ell\geq 0$ such that $V_{\ell}=G$. 
\end{definition}

\begin{fact} Let $G\neq \1$ be a finite group. \label{socle-fact}
\begin{enumerate} 
\item \label{socisproduct} The socle  $\soc(G)=\langle N_1 \ldots N_v \rangle$ is generated by the distinct minimal normal subgroups $N_i$ of $G$, is the internal direct product of a subset  $\{N_1, \ldots, N_v\}$, and is normal in $G$. 
\\ (Note that by minimality $N_i \cap N_j = \1$ (the trivial subgroup) of $G$ for  $i\neq j$.)
\item $\soc(G)$ is isomorphic to $S_1^{k_1} \times \cdots \times S_m^{k_m}$   where $m\geq 1$,  $S_i$ is simple,
$1 \leq k_i$ for $1 \leq i \leq m$. 
\item $\soc(G)$ is a characteristic normal subgroup of $G$. \label{socChar}
\item  $\1 =V_0 \lhd V_1 \lhd V_2 \lhd \cdots \lhd  V_{\ell} =G$ is a characteristic series with each $V_i$ is characteristic  in $G$.  \label{socCharSeries}
\item $\soc(\soc(G))=\soc(G)$ \label{SocIdempotent}
\item $\soc(G/V_i)=V_{i+1}/V_i=\soc(V_{i+1}/V_i)$.  \label{SocQuotients}
\item  \label{SmoothSocSeries}$\sx(G/V_i)=\ell-i$
for $0\leq i \leq \ell$.
\item If $G$ is a simple or a direct product of simple groups,  $\sx(G)=1$. \label{spxsimple}
\end{enumerate}
\end{fact}
\begin{proof}
(1) holds by definition of socle and  Lemma~\ref{ProdMinNorm}(1). (2) follows from (1) and Lemma~\ref{SpanOfGems}.
(\ref{socChar}):  An automorphism $\alpha$ of $G$ send a minimal normal subgroup $M$ of $G$ to a minimal normal subgroup $\alpha(M)$ of $G$. Hence, $\alpha$ permutes the minimal normal subgroups of $G$. Thus, the group the minimal normal subgroups generate, namely $\soc(G)$, is mapped onto itself under automorphisms.  (\ref{socCharSeries}):  
  We show by induction on $i$ that $V_i$ is characteristic in $G$, i.e., $\alpha(V_i)=V_i$.  (In particular $V_i$ is normal since conjugation is an automorphism.)   
For $i=0$, $\alpha(V_0)=V_0$ is trivial.  Suppose the assertion is true for $i$.   Let $w\in V_{i+1}$, we have $h_i(w)\in \soc(G/V_i)$, i.e. $wV_i \in \soc(G/V_i)$. By (\ref{socisproduct}) we may write 
$$ w V_i = w_1 V_i \ldots w_k V_i, $$
for some $w_i\in G$, where the  each $w_j V_i$ lies in a minimal normal subgroup of
$G/V_i$.  Note: $w$ and $w_j$ lie $V_{i+1}$, since $w V_i, w_j V_i \in \soc(G/V_i)$.
Now $\alpha(w_j V_i)$ lies in a minimal normal subgroup of $\alpha(G)/\alpha(V_i)= G/V_i$ by the
induction hypothesis that $\alpha(V_i)=V_i$. Therefore
$\alpha(wV_i)=\alpha(w)V_i$ is a product of elements of minimal normal subgroups of $G/V_i$. 
This shows $\alpha(w)V_i \in \soc(G/V_i)$, so $\alpha(w) \in V_{i+1}$. 
Since $w$ was an arbitrary element of $V_{i+1}$, $\alpha(V_{i+1})\leq V_{i+1}$. Now $\alpha$ is injective and
 $V_{i+1}$ finite, so $\alpha(V_{i+1})=V_{i+1}$, i.e., as $\alpha$ was arbitrary, $V_{i+1}$ is characteristic in  $G$.   So the result follows by induction.

(\ref{SocIdempotent}):
The $\soc(G)$ is the direct product of simple groups $S_1\times \cdots \times S_k$.  Hence each factor is a minimal normal subgroup of $\soc(G)$, hence contained in $\soc(\soc(G))$. This implies the group these $S_i$ generate, namely, $\soc(G)$ is a subgroup of $\soc(\soc(G))$; but the latter is a subgroup of $\soc(G)$.

(\ref{SocQuotients}): By definition $V_{i+1}$ is the inverse image in $G$ of $\soc(G/V_{i})$ under the quotient map,  and so $\soc(G/V_{i})=V_{i+1}/V_i$. So $\soc(\soc(G/V_{i}))=\soc(V_{i+1}/V_i)$, whence by 
(\ref{SocIdempotent}), $\soc(G/V_i)) = \soc(V_{i+1}/V_i)$ too.

(\ref{SmoothSocSeries}) follows directly from the definition of socle length. \\
(\ref{spxsimple}) follows since a direct product of simple groups is its own socle by (\ref{SocIdempotent}).
\end{proof}

A finite simple group has complexity at most 1.  Socle length $\sx$ assigns it complexity 1.

\begin{theorem} Socle length $\sx$ dominates any hierarchical complexity function on finite groups.  \label{dom}
\end{theorem}
\begin{proof}
Let $\c$ be any complexity function on finite groups.
We prove for all finite groups $ \c(G) \leq \sx (G)$ by strong induction on $|G|$.  
If $G=1$, $\sx(G)=\c(G)=0$. 
We have $\soc(G)$ embeds in $G$ with quotient $G/\soc(G)$, therefore by the extension axiom
$\c(G) \leq \c(\soc(G)) + \c(G/\soc(G))$.
Now $\soc(G)$ is a direct product of simple groups $S_1, \ldots \times S_k$, so by the product axiom  $\c(\soc(G)) = \max_{i=1}^k \c(S_i) \leq \1=\sx(\soc(G))$ (by Fact~\ref{socle-fact}(\ref{spxsimple})) since simple groups $S_i$ have $\c(S_i)\leq 1$ (by Lemma~\ref{simplecpx}).
By induction hypothesis, $\c(G/\soc(G)) \leq\sx(G/\soc(G))$.
Now the socle length of $G$ is one more than the socle length of $G/\soc(G)$ by definition of socle length (Fact~\ref{socle-fact}(\ref{SmoothSocSeries})), hence $$\c(G) \leq \sx(\soc(G))+ \sx(G/\soc(G))=1+\sx(G/\soc(G))=\sx(G).$$   \end{proof}

In particular 
\begin{corollary}\label{upperbound}
$\cx(G) \leq \sx(G)$ for all finite groups $G$.
\end{corollary}

\begin{theorem}[Complexity Upper Bounds]\label{upperbounds}

For every finite group $G$,
$$ \mu(G) =\cx(G)\leq \sx(G) \leq \chief(G) \leq \JH(G),$$
where $\chief$ is the length of a chief series for $G$ (which can include one type of simple group in the quotients).

\end{theorem}
\begin{proof}
We have seen $\mu(G)=\cx(G)$ in Theorem~\ref{mucpx} and that $\cx(G) \leq \sx(G)$ in Corollary~\ref{upperbound}.
One can obtain a chief series for $G$ by refining the socle series of $G$: the socle series has a quotients $V_{i+1}/V_i$, each  a product minimal normal subgroups of $V_{i+1}$, by adding appropriate subgroups one at a time between $V_i$ and $V_{i+1}$ one obtains a refined normal series with quotients isomorphic to these minimal normal subgroup. This is a chief series, hence $\sx(G) \leq \chief(G)$. 
Clearly $\chief(G)\leq \JH(G)$ since a single quotient in a chief series has one or more Jordan-H\"older factors. 
\end{proof}

\noindent But does $\sx = \cx$?   If $\sx$ satisfied the complexity axioms, it would follow by maximality of $\cx$ that the two are equal.
We'll now see this is not the case.

\subsection{Socle Length and the Complexity Axioms}

\begin{lemma}\label{SocOfProd}
For a direct product of finite groups $H\times K$, we have $$\soc(H\times K)= \soc(H) \times \soc(K).$$
\end{lemma} 
\begin{proof}
If $M$ is minimal normal in $H$ then $M\times 1$ is minimal normal in $H\times K$; similarly, if $M'$ minimal normal in $K$, then $\1\times M'$ is minimal normal in $H\times K$. It follows that $\soc(H) \times \soc(K)$ is contained in $\soc(H\times K)$. Conversely, if $M$ is minimal normal in $H\times K$, then $M$ projects to a minimal normal subgroup or $\1$ in each factor (Lemma~\ref{Min}), hence $M\leq \pi_H(M)\times \pi_K(H)$, where $\pi_H$ and $\pi_K$ are the projections.  Since $M\neq \1$, either one or both of these projections is minimal normal in $H$ resp.\ $K$.   Thus $M\leq  \soc(H)\times \soc(K)$,  whence $\soc(H\times K)\leq  \soc(H)\times \soc(K)$.   This proves the lemma.
\end{proof}

\begin{proposition}

Socle length $\sx$ satisfies 
the constructability, product, and initial condition axioms on finite groups.
\end{proposition}
\begin{proof}

Socle length satisfies these axioms: 

\underline{Initial condition.} We have $\sx(\1)=0$ by the definition of socle length.

\underline{Constructability axiom.}  Every finite group can be constructed from simple groups by iterated extension from simple groups $S$, namely its Jordan-Hölder factors,  
which have $\sx(S)=1$ by Lemma~\ref{spxsimple}.

\underline{Product axiom.}  Let $G=H\times K$.   
We now use strong induction on $|G|$ to show $\sx(G)=\max(\sx(H),\sx(K))$. 
If $G=\1$, then $\sx(G)=0$ and the assertion holds.  Otherwise,
by Lemma~\ref{SocOfProd}, $$\soc(G)=\soc(H\times K) =\soc(H)\times \soc(K),$$ 
so $$G/\soc(G)\cong (H \times K)/(\soc(H)\times \soc(K)) \cong H/\soc(H) \times K/\soc(K),$$
where the last isomorphism holds since the socles are normal in the respective factors (e.g. \cite[Theorem 2.30]{Rotman}).
By induction hypothesis, $$\sx(G/\soc(G))=\max(\sx(H/\soc(H)), \sx(K/\soc(K))).$$ Therefore, by Fact~\ref{socle-fact}(\ref{SmoothSocSeries}) applied to $G$, $H$ and $K$, 
$$\sx(G)=1+\sx(G/\soc(G))=1+ \max(\sx(H/\soc(H)), \sx(K/\soc(K)) $$
$$=\max(1+\sx(H/\soc(H)), 1+\sx(K/\soc(K)))=\max(\sx(H), \sx(K)),$$
completing the induction.
 \end{proof}

\begin{counterexample}[Extension Axiom Fails for Socle Length]\label{pgroup}
Consider $$G= \langle (1,2), (3,4), (5,6), (7,8), (1,3,5,7)(2,4,6,8)\rangle.$$
It is not hard to see that $|G|=64$ and $G\cong {\mathbb Z}_2 \wr {\mathbb Z}_4$, the wreath product of cyclic groups of order 4 and 2, projecting onto the $\Z_4$ factor.\footnote{To see the isomorphism, using depth-preserving maps of a two-level tree branches 4 ways from the root to depth 1, and 2 ways to 8 nodes at depth 2. Then
the first 4 generators each fix the nodes at depth 1 and permute the two nodes under a node at depth one. The final generator cyclically permutes the depth 1 nodes but maps the nodes at depth two trivially. }

The center $Z(G)$ of $G$ is a two-element group $$Z(G)=\langle (1,2)(3,4)(5,6)(7,8)\rangle$$
and intersects any normal subgroup $N\neq \1$ of $G$ nontrivially since $G$ is a $p$-group 
(as is well-known $Z(G)\cap N\neq \1$ for $G$ a $p$-group since $N$ is a union of conjugacy classes, but of order a power of $p$, hence there
most be conjugacy classes other than $\1$ of size 1, i.e., other central elements in $N$, since the sum of the sizes of conjugacy classes is $|N|\cong 0$ modulo $p$.). 
It follows that this two-element $Z(G)$ is the unique minimal normal subgroup of $G$.  Thus $\soc(G)=Z(G)$.
Direct computation shows $\sx(G)=4$, with socle series\\ $\1\lhd V_1 \lhd V_2 \lhd V_3 \lhd V_4= G,$ where
$V_1=Z(G)$, $V_2= \langle (3,4)(7,8), (1,2)(3,4)(5,6)(7,8)\rangle$, $V_3=\langle (1,5)
  (2,6)(3,7)(4,8), (5,6)(7,8), (3,4)(7,8) \rangle, $
  with successive quotients isomorphic to $\Z_2$, $\Z_2$, $\Z_2\times \Z_2$, $\Z_2 \times \Z_2$.

On the other hand consider $N=\langle (3,4)(7,8), (1,2)(5,6) \rangle.$  
One easily checks $N$ is normal in $G$ and is a Klein 4-group.  Since it is a direct product of simple groups, $\sx(N)=1$.  Direct computations show $Q=G/N$ is isomorphic to a semidirect product $(\Z_4 \times \Z_2) \rtimes \Z_2 \cong \langle (2,6)(5,8), (1,2,3,5)(4,6,7,8)\rangle$ with socle
$Q_1=\langle (1,4)(2,6)(3,7)(5,8), (1,3)(2,5)(4,7)(6,8) \rangle \cong \Z_2 \times \Z_2,$ and
$Q/Q_1 \cong \Z_2\times \Z_2$, so $\sx(G/N)=2$.
Now $$\sx(G)=4 \not\leq 1+2 = \sx(N)+\sx(G/N),$$ so this is a counterexample to the extension axiom for $\sx$. 
\end{counterexample}  

\begin{counterexample}
The smallest counterexample to extension for $\sx$ is the 32-element group 
$H=\langle (1,2)(3,5)(4,6)
      (7,8), (2,5,6,8)(3,7), (2,6)(5,8)\rangle$, whose 
 socle length is $3$ with socle series is $\1 \lhd V_1 \lhd V_2 \lhd V_3=H$. 
 The center of $H$ is a unique minimal normal subgroup $$V_1=\soc(H)=Z(H)=\langle (1,2)(3,5)(4,6)(7,8)\rangle\cong \Z_2,$$
      $$V_2=\langle (2,6)(5,8), (1,7)(2,8)(3,4)(5,6), (1,4)(2,6)(3,7)(5,8)\rangle\cong \Z_2\times\Z_2\times \Z_2,$$ whose quotient by $\soc(H)$ is $\Z_2\times \Z_2$, with $H/V_2\cong \Z_2\times \Z_2$. So $\sx(H)=3$.
      However, since
	$V_2\cong \Z_2\times\Z_2\times \Z_2$ and $H/V_2\cong\Z_2\times \Z_2$, shows $\cx(H)\leq 2$,
	 but $\cx(H)\neq 1$ since $H$ is not the direct product of $\Z_2$'s as it has an element of order $4$, whence $\cx(H)=2$.
	Now $\cx(H)=2<3=\sx(H)$. Note also that the $\sx(H)$ is {\em not} bounded above by $\sx(V_2)+\sx(H/V_2)=1+1=2$.
\end{counterexample}
Since $\mu=\cx$ satisfies the extension axiom by Prop.~\ref{mu-axioms}, we have this
\begin{corollary}\label{spx-no-ext}
$\sx$ does not satisfy the extension axiom.  Hence, $\cx\neq \sx$.
\end{corollary}
\begin{proposition}

$\sx$ satisfies the quotient and normal subgroup axioms.
\end{proposition}

\begin{lemma} Let $G$ be a finite group and $N$
a normal subgroup of $G$.   
Then, the socle characteristic series of $N$ and  of $G/N$
are each no longer than that of $G$.
That is, (1) $\sx(N) \leq \sx(G)$  and (2) $\sx(G/N) \leq \sx(G)$.
\end{lemma}

\begin{proof} 
\noindent
 (1) \underline{Normal Subgroup Property}.    
Let  $G$ be finite group with socle
characteristic series as defined in Section~\ref{socleSeries} 
$$ \1=V_0\lhd V_1 \lhd \cdots \lhd V_{\ell} \lhd V_{\ell+1}=G.$$ 
For $N$ a normal subgroup of $G$, let
$$1 =W_0 \lhd W_1 \lhd \cdots \lhd W_{m} =N$$
be the socle characteristic series for $N$.   We claim
  
\begin{claim}\label{Nclaim}
For all natural numbers $i$, 
$V_i \cap N \leq W_i.$
\end{claim}
\noindent (By definition of the socle characteristic series, $W_i=N$ for $i\geq m$ and $V_i=G$ for $i\geq \ell$.)

Since $V_{\ell}=G$, from the claim it follows that $N=V_{\ell}\cap N \leq W_{\ell}$.  Thus $N=W_{\ell}$, but by definition of the socle characteristic series for $N$, $m$ is the least natural number such that
$W_m=N$, so $m \leq \ell$, i.e., $\sx(N)\leq \sx(G)$.

So to prove (1), it suffices to  prove Claim~\ref{Nclaim}. 

\noindent {Proof of Claim 1 by Induction}:
The assertion is trivial for $i=0$.

Suppose it is true up to $i$
 $$ V_i\cap N \subseteq W_i.$$
 Then $V_i \cap N \lhd W_i$ since $V_i \lhd$ and $N \lhd G$.    We must show $V_{i+1} \cap N \lhd W_{i+1}$.
 Consider the natural map $\varphi: G\sur G/V_i = \varphi(G)$.  Then $\varphi(N)\lhd G/V_i$, since $\varphi$ is surjective.
We shall denote $\varphi(N) =NV_i/V_i$ by $N/V_i$. 
By Splitting Lemma~II~(Lemma~\ref{SpanOfGems2}(1)), 
$$\soc(G/V) \cap N/V_i= \tilde{N}_1\cdots \tilde{N}_k$$ for direct factors $\tilde{N}_i$ minimal normal in $\varphi(N)$.  
By induction hypothesis, $V_i\cap N \lhd W_i$, so we have a natural surjective morphism:
$$\psi: \frac{N}{V_i\cap N}{\sur} \frac{N}{W_i}.$$
Each $\tilde{N}_i$ maps to $\1$ or a minimal normal subgroup of $N/W_i$ under $\psi$ by Lemma~\ref{Min}.
Hence $\tilde{N}_i$ maps to $\soc(N/W_i)=W_{i+1}/W_i$.
Therefore all of $\soc(G)\cap N/V_i$ maps to $W_{i+1}/W_i$. Now $\soc(N/V_i)$ is generated by minimal normal subgroups of $N/V_i$, and by Splitting Lemma II (Lemma~\ref{SpanOfGems2}(1)), so is $\soc(G/V_i) \cap N/V_i$. 
Now $\soc(G/V_i)=V_{i+1}/V_i$.  Suppose $x\in V_{i+1} \cap N$.   
Then $xV_i \in NV_i$,  and $x V_i = n_1 V_i \ldots n_kV_k$ where each $n_j V_i \in \tilde{N_j} (1\leq j \leq k)$. Applying $\psi$ we have 
 $$x V_i = n_1 V_i \ldots n_kV_i \mapsto n_1 W_i \ldots n_k W_i \in \soc(N/W_i),$$
 whence $$xW_i = n_1 \ldots n_k W_i \in W_{i+1}/W_i.$$
 with $n_1\ldots n_k \in W_{i+1}$. Therefore,
 $(n_1\ldots n_k)^{-1}x \in W_i \lhd W_{i+1}$.
 It follows that $x \in W_{i+1}$. Since $x$ was an arbitrary element of $V_{i+1}\cap N$, this completes the induction step. Therefore the assertion holds for all $i\geq 0$.  This proves Claim 1, hence (1) is proved.\\

\noindent
 (2) \underline{Quotient Axiom}.    Similarly, let 
$$\1 =W_0 \lhd W_1 \lhd \cdots \lhd W_{n} = H$$
be the socle characteristic series for $H=G/N$.  Using the notation $V_i/N$ to denote the homomorphic image of $V_i$ under the quotient map $G\sur G/N$, we claim
\begin{claim} \label{Qclaim}with each $M_i'\cap \soc(N)=\1$, $s\geq 0$
Let $\varphi:G \sur H$ be a surjective homomorphism.
For all $i\geq 0$,
$\varphi(V_i) \leq W_i$.
 \end{claim}
 
 \noindent (By definition, $W_i=G/N$ for $i\geq n$ and $V_i=G$ for $i\geq \ell$.)

Similarly to be before, since $V_{\ell}=G$, from the claim it follows that $\varphi(V_{\ell})  \leq W_{\ell}$.  Thus $G/N=W_{\ell}$, but by definition of the socle characteristic series for $G/N$, $n$ is the natural number such that
$W_n=G/N$, so $n \leq \ell$, i.e., $\sx(H)\leq \sx(G)$.

So to prove (2), it suffices to  prove Claim~\ref{Qclaim}. 

\noindent {\em Proof of Claim 2.}  We proceed by induction. The assertion is trivial for $i=0$, since $\varphi(V_0)=\1 \leq \1 =W_0$.  
Suppose $i\geq 0$ and $\varphi(V_i)\leq W_i$.
Then $\varphi$ is surjective so $\varphi(V_i)\lhd H$ and so $\varphi(V_i) \lhd W_i$.
We must show $\varphi(V_{i+1})\leq W_{i+1}$.

Consider a minimal normal subgroup $M$ of $G/V_i$.
By Lemma~\ref{Min},  $M$ is either trivial or a minimal normal subgroup of $G/V_i$.  Since $\varphi(V_i) \leq W_i$ by induction hypothesis, there is well-defined surjective morphism from $H/V_i$ to $H/W_i$; and by Lemma~\ref{Min},
$\varphi(M)$ is either trivial or a minimal normal subgroup in $H/W_i$.
Hence $M$ maps to a $\soc(W_{i+1}/W_i)=\soc(H/W_i)$.
Since $M$ is an arbitrary minimal normal subgroup of $G/V_i$, and $\soc(H/W_i)$ is generated by minimal normal subgroups of $H/W_i$, it follows (using Fact~\ref{socle-fact}(\ref{SocQuotients})) that $\soc(G/V_i)=\soc(V_{i+1}/V_i)=V_{i+1}/V_i$ maps homomorphically to $\soc(H/W_i)=\soc(W_{i+1}/W_i)=W_{i+1}/W_i$.

Now let $v$ be arbitrary element of $V_{i+1}$. We have that $$vV_i 
\mapsto \varphi(v)W_i $$
under the composite morphism from $G/V_i$ to $H/W_i$.
Therefore $\varphi(v)W_{i+1}$ lies in the socle of $H/W_i$. By definition of $W_{i+1}$, $\varphi(v)\in W_{i+1}$.  Therefore $\varphi(V_{i+1})\leq W_{i+1}$, and the result of Claim 2 follows by induction. This proves (2).
\end{proof} 

\section{Other Complexity Functions: Lower and Upper Bounds}

\subsection{Uncountably Many Complexity Functions} 
Here we show for any collection of finite simple groups we obtain a distinct maximal  complexity function
assigning $1$ to these simple groups and $0$ to all others.

\begin{definition}
Let $\S$ be a collection of finite simple groups.   Define the {\em characteristic complexity function with respect to $\S$} by 

   $$\chi_{\S}(G)
   = \begin{cases}
      1 & \mbox{if some Jordan-H\"{o}lder factor of $G$ lies in } {\S} \\
      0      & \mbox{ otherwise.}
    \end{cases}$$

\end{definition}

It is immediate to check 
\begin{lemma} 
The  function  $\chi_{\S}$ satisfies the complexity axioms.
\end{lemma}

\begin{examples} We have the following complexity functions:
\begin{enumerate} 
\item 
 $z=\chi_{\emptyset}$
\item Let $\SIMPLE$ be the finite simple groups. $\delta=\chi_{\SIMPLE}$. 
\item 
Let ${\SNAG}$ be the collection of all finite simple non-abelian groups. 
Then $\chi_{\SNAG}(G)=0$ if and only if $G$ is solvable.
\item Let ${\mathscr Z}$ be all simple groups $\Z_p$ of prime order.  Then
  $\delta=\max(\chi_{\SNAG},\chi_{\mathscr Z}).$
  \end{enumerate}
\end{examples}

\begin{definition}
If $\S$ is collection of finite simple groups. 
$\chi_{\S}^*$ is the maximal function pointwise dominating all complexity functions which
assign zero to all finite simple groups not in $\S$. That is, 
$$\chi_{\S}^*(G)=
\max\{\c(G): \c \mbox{ is a complexity function  with  $\c(K)\leq \chi_{\S}(K)$  for all $K$ simple}\}.$$
   
\end{definition}
As for $\cx$, which is $\chi^*_{\SIMPLE}$, one shows that $\chi_{\S}^*$ exists and is the unique
complexity function assigning $1$ to members of $\S$ and $0$ to other finite simple groups.

For a subnormal series  for a group $G$ in which each quotient $G_{i+1}/G_i$ is the direct product of simple groups, we say {\em the $i$th level of the series has members in $S$} if the $i$th quotient group has at least one factor in $\S$.
Let $\mu_\S(G)$ be the least number of levels over all such series for which the series has members of $\S$.

\begin{theorem}  For any set $\S$ of simple groups,
$\mu_{\S}=\chi^*_{\S}.$
\end{theorem}
\begin{proof}
Just as in the proofs the $\mu=\cx$ by Prop.~\ref{mu-axioms} and Theorem~\ref{uniqueCpx}, one proves $\mu_
\S$ satisfies the axioms and dominates all complexity functions $\c$ with $\c(K)\leq \chi_{\S}(K)$ for all $K$ simple. The only subtle point is 
the product axiom, which is verified  in the Lemma~\ref{alignment}. \end{proof}

\begin{lemma}[Alignment Lemma]\label{alignment}
The product axiom holds for $\mu_\S$. 
\end{lemma}
\begin{proof} Let $H$ have a subnormal series with quotients direct products of simple groups
$$\1  =H_0 \lhd H_1 \lhd \cdots \lhd H_a=H,$$ with $h\leq a$ levels having factors in $\S$.
Similarly,  let $K$ have subnormal series of this type with $k \leq b$ levels having factors in $\S$
$$\1 =K_0 \lhd K_1 \lhd \cdots \lhd K_b=K.$$ 
Suppose these series have least possible $h$ and $k$, respectively.  That is, $\mu_\S(H)=h$ and $\mu_S(K)=k$.
If $H\times K$ had $\mu_\S(H\times K) < \max\{h,k\}$, then it would have a series projecting onto series for $H$ and $K$ with less than $h$ and $k$ levels with factors in $\S$, respectively, a contradiction (just as in the proof of Prop.~\ref{mu-axioms}), so $\mu_\S(H\times K) \geq \max\{\mu_\S(H),\mu_\S(K)\}$.

Next we claim $H\times K$ has subnormal series with quotients direct products of simple groups having exactly $\max\{h,k\}$ levels with factors in $\S$.  

Let $i_1 < \ldots < i_h$ be the indices in $[0,a-1]$ such that $H_{i_\ell +1}/H_{i_\ell}$ has factors in $\S$.
Similarly, let $j_1 < \ldots < j_k$ be the indices in $[0,b-1]$ such that $H_{j_\ell+1}/H_{j_\ell}$ has factors in $\S$.
We have $H_{i_\ell} \times K_{j_\ell } \lhd H_{i_\ell+1} \times K_{j_\ell+1}$, where if $h<k$ we take $H_{i_\ell}=H_{i_h}$ for all $\ell>h$, and similarly if $h>k$ we take $K_{j_\ell}=K_{j_k}$ for all $\ell>h$. 
Then the groups $$(H_{i_\ell+1} \times K_{j_\ell+1})/(H_{i_\ell} \times K_{j_\ell})\cong 
(H_{i_\ell+1}/ H_{i_\ell}) \times (K_{j_\ell+1} / K_{j_\ell})$$
are each a direct product of simple groups with factors coming  factors in $\S$.
It is easy to see that the normal inclusions $$H_{i_\ell} \times K_{j_\ell } \lhd H_{i_\ell+1} \times K_{j_\ell+1}$$ for $1\leq \ell \leq \max\{h,k\}$ can be extended to a subnormal series for $H\times K$ in which these are the only quotients with factors in $\S$:  One simply inserts products of factors from the above series for $H$ and $K$, e.g., include the subgroups $H_0 \times K_j$ for $j \leq j_1$, followed by $H_i\times K_{j_1}$ for $i < i_1$, and so on inserting product groups between $H_{i_\ell+1} \times K_{j_\ell+1}$ and $H_{i_{\ell+1}} \times K_{j_{\ell+1}}$ 
by first increasing the factors from the $H_i$ and then the ones from the $K_j$. Finally insert $H_{i_h+1}\times K_j$
for all $j$ with $b\geq j>{j_k+1}$, if any, followed by $H_i \times K$ for all $i$ with $a\geq i > {i_h+1}$, if any. This yields a subnormal series from $1$ to $H\times K$ having $\max\{h,k\}$ levels with factors in $\S$. 
This shows $\mu_\S(H\times K)\leq \max\{\mu_\S(H), \mu_\S(K)\}$ and completes the proof.\end{proof} 

\begin{corollary}
 There are uncountably many pairwise distinct complexity functions
on finite groups. In particular, if $\S,\S' \subseteq \SIMPLE$ are distinct then $\mu_\S\neq \mu_{S'}$.
\end{corollary}
\begin{proof}
$\mu_{\S}$ and $\mu_{\S'}$ disagree on each member of the symmetric difference $(\S\cup \S') \setminus (\S\cap \S')$ which is nonempty
since for $\S,\S' \subseteq \SIMPLE$ with $\S\neq \S'$.  Hence $\mu_\S \neq \mu_{\S'}$. Since there are countably many isomorphism classes of finite simple groups comprising $\SIMPLE$, this yields a distinct complexity function for each of the uncountably many subsets $\S$
of $\SIMPLE$.  \end{proof}

Since $\mu_S$ satisfies the complexity axioms and $\cx= \mu=\mu_\SIMPLE$ is the maximal complexity function on finite groups, we conclude:
\begin{theorem}
Let $\S \subseteq \SIMPLE$, the $\mu_\S$ is a lower bound for hierarchical complexity $\cx$. 
\end{theorem}

\subsection{Solvable Groups and the Embedding Axiom}\label{solv}

Next we consider complexity functions on finite solvable groups.
\begin{theorem}
There exists a unique maximal complexity function  $\cxsolv$ on finite solvable groups satisfying the axioms.
\end{theorem}
\begin{proof} By the same reasoning as in  Theorem~\ref{uniqueCpx}, the exists a unique maximal complexity function on solvable groups. \end{proof}

\begin{lemma}  For a finite solvable group $G$ and prime $p$, let $\log_p(G)$ be the greatest natural number $n$ such that $G$ has an element of order $p^n$.   Then $\log_p$ is a complexity function on solvable groups.
\end{lemma}
Caveat: $\log_p$ is not a complexity function on all finite groups since it can assign value greater than one to simple non-abelian groups. Indeed, by Lemma~\ref{EmbedInAlternatingGroup}, we can embed $\Z_{p^n}$ into a simple group $K_n$ results in $\log_p(K_n) \geq n$ for any positive $n$.

\begin{proof}

\noindent\underline{Product}.  If $G=H\times K$ and $g=(h,k)$, then the order of $g$ is the least common multiple the orders of $h$ and$k$.
Therefore, if $g\in G$ has order $p^n$, then $h$ and $k$ both have order dividing $p^n$, so they have orders $p^a$
and $p^b$, with $\max(a,b)=n$.  This entails that $\log_p(G)=\max\{\log_p(H),\log_p(K)\}$.

\noindent\underline{Extension}.  Suppose $a=\log_p(Q)$, $b=\log_p(N)$, and $G$ is an extension of $Q$ by $N$. Then $G$ embeds in a wreath product $N\wr Q$.
Let $g$ in $G$ be an element whose order  is a power of $p$, then in the embedding,
$g=(\n,q)$, $q\in Q$, $\n \in N \times \cdots \times N$ (the $|Q|$-fold direct power of $N$).
 We have $\log_p(N)=\log_p(N^Q)$ (since the product axiom holds). 
Now $$g^{p^a}=(\n',q^{p^a})=(\n',1)$$ for some $\n'\in N^Q$, whence
$$g^{p^{(a+b)}}=(g^{p^a})^{p^b}=(\n',1)^{p^b}=(\n'^{p^b},1)=(1,1)=1 \in G.$$ 
This shows the order of $g$ is at most $p^{a+b}$. This shows $$\log_p(G)\leq \log_p(N)+\log_p(Q),$$ i.e., the extension axiom holds.

\noindent\underline{Quotient}.  If $g$ has order $p^n$ in $G$, its image in any quotient $Q$ has order dividing $p^n$, whence
$\log_p(G)\geq \log_p(Q)$.

\noindent\underline{Normal Subgroup Axiom / Subgroup Axiom}.  If $N$ is a subgroup of $G$, whether normal or not, and $n\in N$ has order $p^n$, so $G$ also contains this element, hence $\log_p(G)\geq \log_p(N)$.

\noindent\underline{Constructability}.  Every finite solvable group can be constructed by iterated extension from simple cyclic groups $\Z_q$ of  prime order $q$.
Since $\log_p(\Z_q)=1$ for $p=q$ but is $0$ otherwise, this shows the constructability axiom holds.

\noindent\underline{Initial Condition}. Finally, the element of largest prime power in the trivial group has order $1=p^0$, so $\log_p$ of the trivial group is $0$.
\end{proof}

\begin{remark}
Note that $\log_p(S)$ can be greater than 1 if $S$ is a simple non-abelian group. Therefore $\log_p$ is {\em not} is complexity function on all finite groups as it fails to satisfy the constructability axiom.
\end{remark}

\begin{corollary}[Some Complex Bounds]
The following bounds hold related  the complexity $\cx_\solv$ of solvable groups. 
\begin{enumerate} 
\item $\log_p$ is unbounded since $\log_p(\Z_{p^n})=n$ for cycle groups of order $p^n$.
\item  For a finite solvable group $G$,
$$\cxsolv(G) \geq \sup_{\mbox{$p$ prime}} \log_p(G)
   = \max_{\mbox{$p$ prime divisor of $|G|$}} \log_p(G).$$
\item The complexity function $\cxsolv$ is unbounded. 
\item  $\sx \geq \cxsolv$.
\end{enumerate}
\end{corollary}
\begin{proof} (1) follows from the definition of $\log_p$.  (2) follows since $\cxsolv$ is a maximal complexity function on solvable groups so dominates each complexity function $\log_p$. (3) follows from (2) and (1). 
(4) follows since $\sx$ bounds any complexity function on solvable groups from above (same proof as for $\sx$ bounding $\cx$ on all finite groups). 
\end{proof}

\begin{theorem}\label{cpxsolv-subgroup-ax}
\enumerate
\item If $G$ is a solvable group, then $\cx(G)=\cx_{Solv}(G).$
\item $\cxsolv$ satisfies the embedding axiom:
 $\cxsolv(H)\leq \cxsolv(G)$ for $H$ a subgroup of $G$.
 \item $\sx$ restricted to solvable groups satisfies the embedding axiom.  If $G$ is solvable and $H$ a subgroup, then $\sx(H) \leq \sx(G)$. 
\end{theorem}

\begin{proof}  
(1) Recall $\cx(G)=\mu(G)$  the minimal length of a composition series for $G$ with all quotients direct products of simple groups.
The same proof that $\mu(G)$ is a maximal complexity function on finite groups works to prove $\mu(G)$ restricted to solvable groups
is a maximal complexity function on solvable groups.
Since $\mu$ restricted to solvable groups satisfies the complexity axioms (recall the embedding axiom was not required) and $\cxsolv$ is a maximal complexity function on solvable groups,  
$$\cxsolv(G) = \mu(G).$$  
(2) Consider a subnormal series of length $\mu(G)$ for $G$,
$$1 = V_0 \lhd V_1 \lhd \ldots \lhd V_n = G,$$
witnessing $\mu(G)=n$ with $V_{i+1}/V_i$ a direct product of simple abelian groups.
Intersection with subgroup $H$ yields a subnormal series for $H$:
$$1 = V_0 \cap H \lhd V_1 \cap H  \lhd \ldots \lhd V_m \cap H = G \cap H =H.$$
By the 2nd of Noether's isomorphism theorems, this  has quotients 
$$(V_{i+1} \cap H)/(V_i \cap H) \cong  (V_{i+1}\cap H)V_i/V_i \,
\leq \,V_{i+1}/V_i,$$
where the inclusion holds since $(V_{i+1}\cap H)V_i \leq V_{i+1}$  for each $i.$
Since $V_{i+1}/V_i$ is a direct product of abelian simple groups,  so is its subgroup the quotient from the series for $H$.   This shows $\mu(H)\leq \mu(G)$.
The proof for $\sx(H)\leq \sx(G)$ is the same, but starts with the socle characteristic series of $G$. 
\end{proof}

By  Theorem~\ref{cpxsolv-subgroup-ax}, when restricted  to solvable groups, the hierarchical group complexity and socle length satisfy the \\

\underline{Subgroup Property/ Embedding Axiom}: 

If $H$ is a subgroup of $G$, then $\c(H)\leq \c (G)$.\\

\noindent
The importance of the theorem is that when working with solvable groups, we can use the embedding axiom to compute their complexity and socle length.

\subsection{Hierarchical Complexity of an Iterated Wreath Product}

\begin{example} \label{Zpwr} 
We compute, for $p$ prime, and any $n$, the complexity of the $n$-fold wreath product of $k_i$-fold direct products of a simple abelian group with itself,
$$W=\wr_{i=1}^n (\prod_{i=1}^{k_i} \Z_p)=(\Z_p)^{k_n}\wr \cdots \wr (\Z_p)^{k_1},$$
where $k_i>0$ ($1\leq i \leq n$).
Namely, $\cx (W) = n.$ Also, $\cx(\Z_{p^n})=n$.
\end{example}
\begin{proof} From example~\ref{ZpnExample}, we know $\cx(\Z_{p^n}) \leq n$. Since $\Z_{p^n}$ is solvable, $\log_p$ gives a lower bound for  $\c_\Der$, 
$$n=\log_p(\Z_{p^n}) \leq \cx_\solv(\Z_{p^n}) = \cx(\Z_{p^n})$$ by Theorem~\ref{cpxsolv-subgroup-ax}.
Therefore $\cx(\Z_{p^n})=n$.
The iterated wreath product $W$ of direct products of simple groups $\Z_p$ contains a cyclic group of order $p^n$, but not $p^{n+1}$ since
$\log_p$ satisfies the extension and product axioms (applied iteratively to $W$). 
Since the embedding axiom holds for $\cx$  on solvable groups,
$$n=\cx(\Z_p^{n})\leq \cx(W).$$
 On the other hand, $W$
is an $n$-fold iterated extension by products of simple groups, so $\cx(W) \leq n$ by the extension axiom. Hence $\cx(W)=n$.
\end{proof}

\subsection{Hierarchical Complexity of $p$-torsion and other Groups}

\begin{example}\label{ptorsion}
Let $p$ be prime and $G$ be a finite $p$-torsion group, i.e.\ $x^p=1$ for all $x\in G$. Then $\cx(G) \leq \sx(G)$ = length of the ascending central series of $G$.  
\end{example}

\begin{proof} $\soc(G)=N_1\cdots N_k$, $N_i$ a minimal normal subgroup of $G$.  Each $N_i \cap Z(G) \neq \1$, since $G$ is a $p$-group.\footnote{This well-known fact is shown in Counterexample~\ref{pgroup}.} By minimality $N_i$,  $N_i\cap Z(G)=N_i$.  Thus every minimal normal subgroup is contained the center of $G$. Therefore $\soc(G)\leq Z(G)$. Conversely if $1\neq x\in Z(G)$, then 
$\langle x \rangle \cong \Z_p$ is minimal normal in $G$.  
Therefore $Z(G)=\soc(G)$. Therefore $Z(G/V_i)=\soc(G/V_i)=V_{i+1}/V_i$ in the socle series, so it coincides with the ascending central series.
The conclusion then follows from Theorem~\ref{dom}.\end{proof}

\begin{open}[Complexity of Finite Prime-Exponent Burnside Groups]
What is the hierarchical complexity of
 $B=B(k,p)$ a largest finite $k$-generated group of exponent $p$ prime?

From example~\ref{ptorsion} and  Theorem~\ref{lowerbounds} on lower bounds proved below, for any $p$-torsion group $B$, 
 $$\Der(B) \leq \cx(B) \leq \sx(B)= \mbox{length of ascending central series of $B$}.$$
 Since the derived series for $B$ has abelian quotients $Q_i$ and each of these quotients is a homomorphic image of $B$, it follows that $Q_i$ is an abelian $p$-torsion group, hence necessarily of the 
 form $\Z_p\times \cdots \times \Z_p$, a span of gems, whence $\cx(B)\leq \Der(B)$.
 Since $\Der(B) \leq \cx(B)$, it follows that the derived series of $B$ gives a decomposition of $B$ at its hierarchical complexity.\footnote{Note it does {\em not} necessarily follow that there is not another
 subnormal series showing $\cx(B)$ has this value, since the quotient factors at complexity need not be unique (cf.~Counterexample~\ref{FactorsNotUnique}). }
 Therefore $\Der(B)=\cx(B)$, for any $p$-torsion group $B$.  
 
 The same argument shows that if the factors of
 $n$ contain no square of a prime,  then any finite $n$-torsion solvable group has $\Der(B)=\cx(B)$:  since the quotients of the derived series are abelian and $n$-torsion, they must therefore be direct products of simple cyclic groups by the fundamental structure theorem for finite abelian groups.\footnote{Similarly, the assertion $\Der(G)=\cx(G)$ holds for any finite group $G$ with square-free order, even without the assumption of solvability of $G$:  By a classical result of Frobenius~\cite[Sec.~4]{Frobenius1893}, a finite group whose order is the product of distinct primes is necessarily solvable, so the derived series terminates at $\1$ and has abelian $n$-torsion quotients which again must be spans of gems.  (See exposition  of Frobenius's result in \cite{Ganev2010}.)}
\end{open}
\begin{open}
For which finite groups $G$, does $\sx(G)=\cx(G)$?

From Theorem~\ref{upperbound}, $\cx(x) \leq \sx(G)$, and this bound is sharp since, e.g., $\sx(\Z_{p^n})=\cx(\Z_{p^n})$:
For  $\Z_{p^n}=\langle x \rangle$, we have $\soc(\Z_{p^n})=\langle x^{p^{n-1}} \rangle\cong \Z_p$, whence by induction $\sx(\Z_{p^n})=n$. 
By example~\ref{Zpwr}, $\cx(\Z_{p^n})=n$ too. 
\end{open}

\subsection{Derived Complexity, Fitting Complexity and Solvability Complexity}\label{der-fit-solv-cpx}

Here we recover the concepts of derived length and Fitting height within the framework of group complexity axioms.
We show that the usual derived length of solvable groups and the the usual Fitting height of solvable groups are restrictions of more general complexity functions on all finite groups, and give lower bounds on hierarchical complexity. 

\noindent To obtain derived length we can add to the complexity axioms the axiom:\\

(Com) Commutative groups have complexity at most 1.\\

\noindent To recover Fitting height,
we add the axiom:\\

(Nil) Nilpotent groups have complexity at most 1.\\

\noindent For solvability, we can have an analogous axiom:\\

(Solv) Solvable groups have complexity at most 1.\\ 

Since finite abelian groups are nilpotent, and nilpotent groups are
solvable (e.g.\ \cite[Ch.~9 \& 10]{Hall}), we have the following implications: \\
$$\mbox{(Solv)} \Rightarrow \mbox{ (Nil) }\Rightarrow \mbox{ (Com)}.$$

\begin{theorem}\label{variantMaxCpxFunctions}
There exist  unique maximal complexity functions $\cx_\Der$, $\cx_\Fit$ and $\cx_\Solv$  on finite groups satisfying the complexity axioms together with the {\rm (Com)}, {\rm (Nil)} and {\rm (Solv)} axioms, respectively. 
\end{theorem}
\begin{proof}
The proof follows same reasoning as in Theorem~\ref{uniqueCpx} but considering only complexity functions also satisfying the added axiom.
\end{proof}

Note that the normal subgroup property is {\em not }assumed, but will be found to hold nonetheless, as we shall see from alternative characterizations of these functions (just as for $\cx$ and $\mu$).

Starting from the notion of a span of gems (or m\={a}l\={a} or necklace of simple groups), i.e., a finite group  that is the direct product of  simple groups, we have the following increasingly general concepts: 
\begin{definition}\begin{enumerate}
    \item  A {\em derived necklace} is a finite group $G$ that is the direct product of abelian groups and simple groups, i.e., $$G \in \Span(\mbox{Abelian groups} \cup \SIMPLE).$$
\item A {\em Fitting necklace} is a finite group that is the direct product of nilpotent groups and simple groups, i.e., $$G \in \Span(\mbox{Nilpotent groups} \cup \SIMPLE).$$
\item A {\em solvability necklace} is a finite group that is the direct product of solvable groups and simple groups, $$G \in \Span(\mbox{Solvable groups} \cup \SIMPLE).$$
\end{enumerate}
\end{definition}
Note that in these definitions, zero or more simple factors are permitted, also the abelian, nilpotent or solvable groups, respectively, may be trivial.  

\begin{fact}\label{product}
\begin{enumerate}
    \item 
Since the direct product of abelian groups is abelian, each derived necklace is the direct product of a single abelian group and (zero or more) simple non-abelian groups. 
\item 
Similarly, since the direct product of nilpotent groups is nilpotent and simple abelian groups are nilpotent, each Fitting span of gems is the direct product of a single nilpotent group and  simple non-abelian groups. 
\item A solvability necklace is the direct product of a single solvable group and simple non-abelian groups, since solvable groups are closed under product and abelian simple groups are solvable. 
\end{enumerate}
\end{fact}

One has the following consequences of the SNAG rigidity lemma~\ref{SNAGrigidity} and its corollary.
\begin{lemma}\label{VariantSoG}
A normal subgroup of a derived necklace is a derived necklace.  A quotient of a derived necklace is a derived necklace.   A product of derived necklaces is a derived necklace.
Similarly, Fitting necklaces are closed under normal subgroups, quotients and products.  
Solvability necklaces are closed under normal subgroups, quotients, and products. 
\end{lemma}
\begin{proof} 
A derived necklace $G$ is the product of an abelian group $K$ and zero or more SNAGs $H_i$ ($1 \leq i \leq \ell,\ \ell\geq 0$).  Hence, by Lemma~\ref{SNAGrigidity} and its corollary, $N\lhd G$, is a product of some of these SNAGs and $\pi(N)$ where $\pi : G\sur K$ is the projection and $\pi(N)$ is normal in $K$. Since $K$ is abelian so is $\pi(N)$. The assertion about product is follows from Fact~\ref{product} on products of necklaces. Similarly, replacing the word ``abelian'' by ``nilpotent'', resp.\ ``solvable'' in this reasoning shows the analogous assertions about Fitting necklaces and solvability necklaces are also correct.
\end{proof}

Just as for spans of gems, these types of necklaces are not closed under under taking subgroups unless the necklaces have no SNAG factors, i.e., unless the necklaces are solvable groups.\\

\begin{definition}
Now we introduce three functions defined on each finite group $G$: 
\begin{enumerate}
\item $\Der(G)$ is the length of a shortest subnormal series such that the quotients are products of an abelian group and zero or more finite simple non-abelian groups.
\item $\Fit(G)$ is the length of a shortest subnormal series such that the quotients are products of a nilpotent group and zero or more finite simple non-abelian groups.
\item $\Solv(G)$ is the length of a shortest subnormal series such that the quotients are products of a solvable group and zero or more finite simple non-abelian groups.
\end{enumerate}
\end{definition}

\begin{theorem} \label{charCpxFunctions}
We have the following equalities of functions on finite groups: 
\begin{enumerate} 
    \item  $\Der=\cx_\Der$, the unique maximal complexity function assigning 1 to all nontrivial abelian groups. 
  \item $\Fit=\cx_\Fit$ is the unique maximal complexity function assigning 1 to all nontrivial nilpotent groups. 
 \item   $\Solv=\cx_\Solv$ is the unique maximal complexity function assigning 1 to all nontrivial solvable groups.
       \end{enumerate}
\end{theorem}
\begin{proof}
The proof that each of $\Der$, $\Fit$ and $\Solv$ is a complexity function satisfying (Com), (Nil) or (Solv), respectively, follows exactly as for $\mu$ in Proposition~\ref{mu-axioms}, but uses Lemma~\ref{VariantSoG} in place of the Span of Gems Lemma~\ref{SpanOfGems} and Fact~\ref{product} in place of the fact that spans of gems are closed under products. That each one is the unique maximal complexity function $\c_\Der$, $\c_\Fit$ or $\c_\Solv$ satisfying (Com),
(Nil) and (Solv), respectively, is then established in the manner of Theorem~\ref{mucpx} showing their maximality among the appropriate class of complexity functions replacing spans of gems by necklaces of the appropriate type (derived necklaces, Fitting necklaces, resp.\ solvability necklaces). 
\end{proof}

The normal property for each of the three functions is shown just as in Theorem~\ref{cpxNorm}. 
\begin{theorem}[Normal Property]
The finite group complexity functions $\Der$, $\Fit$ and $\Solv$ 
each satisfy the Normal property $\c(N) \lhd \c(G)$ for  $N \lhd G$.
\end{theorem}

Next we obtain a series of complexity lower bounds. 

\begin{theorem}[Complexity Lower Bounds]\label{lowerbounds}
Let $G$ be a finite group, then
$$\Solv(G) \leq \Fit(G) \leq \Der(G) \leq \cx(G)$$
\end{theorem}
\begin{proof} This follows immediately from examining the quotients in a minimal subnormal series of the appropriate kind since a span of gems is a derived necklace, a derived necklace is a Fitting necklace, and a Fitting necklace is a solvability necklace.   \end{proof}

\begin{theorem}
\begin{enumerate}
\item There exist unique  maximal complexity functions defined on solvable groups satisfying
$\Solv_\solv$, $\Fit_\solv$ and $\Der_\solv$  solvability {\rm (Solv)}, nilpotent {\rm (Nil)} and commutative axioms {\rm (Com)}, respectively.
    \item
$\Fit_\solv$ is the restriction of $\Fit$ to solvable groups.
\item $\Der_\solv$ is the restriction of $\Der$ to solvable groups.
\item $\Solv_\solv$ is the restriction of $\Solv$ to solvable groups, and $\Solv_\solv=\delta$
which assigns the value $1$ to all non-trivial solvable groups. 
\item $\Der_\solv$, $\Fit_\solv$ and $\Solv_\solv$ each satisfy the subgroup axiom. 
\end{enumerate}
\end{theorem}
\begin{proof} 
 (1) That these maximal complexity functions on solvable groups exist again follows by the same reasoning as in Theorem~\ref{uniqueCpx}
 adding the appropriate axiom. Arguing just as in the proof that $\cx_\solv(G)=\cx(G)$ (Theorem~\ref{cpxsolv-subgroup-ax}(1)), we have (2) and (3): For finite solvable groups, no simple non-abelian groups can occur as factors of a subnormal series, since otherwise they would occur as Jordan-H\"{o}lder factors, contradicting the solvability.
Hence the definition of $\Fit$ coincides with $\Fit$ on solvable groups, and the definition of $\Der$ coincides with $\Der$ on solvable groups. 

The assertion (4) is trivially true, since a solvable group is the product of itself and zero simple groups. 
(5) That these complexity measures each satisfy the subgroup axiom follows as in Theorem~\ref{cpxsolv-subgroup-ax}(2) replacing ``direct product of abelian simple groups'' in the proof by ``abelian group'', ``nilpotent group'', and ``solvable group'',  respectively. 
\end{proof}

\begin{definition}
The {\em Fitting subgroup} $F(G)$ of a group $G$ is generated by all normal nilpotent subgroups of $G$. By Fitting's Theorem (e.g.~\cite[Theorem~5.2.8]{Robinson}), the product of two normal nilpotent subgroups is nilpotent and normal in G, 
so it follows $F(G)$ is the unique maximal nilpotent normal subgroup of $G$. Let $F_0=\1$,
and inductively define $F_{n+1}=\varphi_i^{-1}(F(G/F_n))$ where $\varphi_i: G\sur G/F_i$ is the natural quotient map.
It is easy to see this sequence terminates at $G$ if and only if $G$ is solvable. 
The {\em Fitting height} of a solvable group $G$ is the least $n$ such that $F_n=G$. 
\end{definition}

We use the following Lemma from \cite[p.~17]{wildbook}. 
\begin{lemma}[Rhodes] \label{RhodesLemma}
The Fitting height of a solvable group $G$ equals
\begin{enumerate}
\item the smallest integer $n$ such that 
$G=N_n \rhd N_{n-1} \rhd \cdots \rhd N_1 \rhd N_0=\1$ with $N_{i+1}/N_i$ nilpotent $(0\leq i <n)$.
\item the least $m$ such that $G$ is a homomorphic image of a subgroup of\\ $M_m \wr \cdots \wr M_1$, where the $M_i$ are nilpotent $(0\leq i <m)$.
\item the least $k$ such that there are surjective morphisms\\
$G=G_k \sur G_{k-1} \sur \cdots \sur G_1 \sur G_0= \1$ whose kernels are nilpotent.
\end{enumerate}
\end{lemma}

\begin{theorem}
Let $G$ be a finite solvable group. 
The derived length of group $G$ is $\Der(G)$.
The Fitting height of $G$ is $\Fit(G)$.
\end{theorem}
\begin{proof}  
Since $G$ is solvable,  simple non-abelian factors cannot occur in the quotients of any subnormal series for $G$. Since the derived series of $G$ is the shortest subnormal series with abelian quotients (e.g., \cite[proof of Theorem~9.2.5]{Hall}), its length coincides with $\Der(G)$ for solvable $G$ by definition of $\Der$.
Lemma~\ref{RhodesLemma}(1) says exactly that $\Fit(G)$ is the Fitting height of~$G$. 
\end{proof}

\subsection{Wreath Product of Abelian Groups and Hierarchical Complexity of Iterated Wreath Products}

\begin{fact} \label{DerWrAbelian}
The derived length of an $n$-fold wreath product of non-trivial abelian groups   $W= H_n \wr \cdots \wr H_1$
is  $\Der(W)=n$.
\end{fact}
\begin{proof}
 One can show this by induction of $n$: $[W,W]$ has trivial coordinate at the top level, and each successive 
derived subgroup has only trivial coordinates at the next level, but non-trivial coordinates below that.
Precisely,  consider a wreath product $G= G_2 \wr G_1$.  Let $w=(1,w_1), n=(n_2,1) \in G_2 \wr G_1$, with $w_1\neq 1$.
Consider the action of the commutator $[w, n]=w^{-1}n^{-1} w n$ on $(y,x) \in G_2 \times G_1$:\footnote{Here we are using the fact that the inverse of $w=(w_2, w_1) \in W$ is
$w^{-1}= (w_2, w_1)^{-1}$ mapping  $(y, x)\in G_2 \times G_1$ to $(y (w_2(x w_1^{-1}))^{-1}, x w_1^{-1})$.} 
\begin{eqnarray*}
(y,x)[w, n] 
& = & (y,x)(1, w_1)^{-1} (n_2, 1)^{-1}(1,w_1)(n_2, 1)\\
& = &  (y,x w_1^{-1}) (n_2, 1)^{-1}(1,w_1)(n_2, 1) \\
& = & (y (n_2(x  w_1^{-1}))^{-1}, x w_1^{-1})(1,w_1)(n_2,1)\\
& = & (y (n_2(x  w_1^{-1}))^{-1}, x w_1^{-1}w_1)(n_2,1)\\
& = & (y (n_2(x  w_1^{-1}))^{-1} n_2(x), x ).
\end{eqnarray*}
This computation shows that $[w,n]\in G_2^{G_1} \rtimes \1 \lhd G_2^{G_1} \rtimes G_1=G_2\wr G_1$. Since  $1\neq w_1 \in G_1$,
for any $x_0\in G_2$,  let $n_2(x_0)=g_2$ but $n_2(x)=1$ for $x\neq x_0$. This yields 
$[w,n] = (n', 1)$ with $n'(x_0)=g_2$, $n'(x_0 w_1)=g_2^{-1}$ and $n'(x)=1$ for $x \not\in \{x_0, x_0 w_1)$. 
In particular, in $[G,G]$ we can obtain any $g_2 \in G_2$ in the $x_0$ position of the direct product $G_2^{G_1} \rtimes \1\cong G_2 \times \cdots \times G_2$. It follows the $[G,G]$ projects onto $G_2$.   
Hence $\Der([G,G])\geq \Der(G_2)$ by the quotient axiom for $\Der$. Therefore $$\Der(G_2\wr G_1) = \Der([G_2\wr G_1,G_2\wr G_1])+1 \geq \Der(G_2)+1.$$ 

We have $\Der(H_1)=1$.  Suppose the proposition holds for $n$-factors, then by induction hypothesis
$G_2 =H_{n+1} \wr \cdots \wr H_2$ has $\Der(G_2)=n$ so $W=G_2 \wr H_1$ has $\Der(W) \geq \Der(G_2)+1=n+1$.
By the extension and product axioms for $\Der$, $\Der(W)=\Der(G_2^{H_1}) + \Der(H_1) = \Der(G_2) +1 = n+1$. 
The result now follows by induction.
\end{proof}

\begin{proposition}[Iterated Wreath Product of Complexity 1 Groups]\label{ItWrCpx1}
Let $W=H_n \wr \cdots \wr H_1$ with each 
 $H_i$ be a commutative group of complexity $1$ ($\cx(H_i)=1$) for $1\leq i \leq n$.
 Then $\cx(W)=n$.  
\end{proposition}

\begin{proof}
Each $H_i$ be a direct product of one or more  simple groups for $1\leq i \leq n$.
Let $W= \wr_{i=1}^n H_i$ be the wreath product.  Then by the extension axiom, 
$\cx(W) \leq \sum_{i=1}^n \cx(H_i)$, but $\cx(H_i)=1$ by the product axiom since simple groups have complexity $1$.
Thus $\cx(W) \leq n$.
Now suppose each of the simple groups involved in the construction of $W$ is cyclic, so that $H_i$ is commutative 
(and hence $W$ is solvable). By Fact~\ref{DerWrAbelian}, it follows
the derived length of $W$ is $n$. 
So $n = \Der(W) \leq \cx(W) \leq n$. Therefore, $\cx(W)=n$. 
\end{proof}

\begin{open}
Let $W$ be a iterated wreath product of $n$ (non-trivial) spans of gems $G_1$, \ldots $G_n$,  for $n \geq 2$.
We know if all the $G_i$ are solvable then $\cx(W)=n$ by Prop.~\ref{ItWrCpx1}. Also, by Prop.~\ref{wreathcpx1}, $\cx(W)=2$ for $n=2$ even if some $G_i$ are simple non-abelian groups.
Does  $\cx(W)=n$ in general? 
\end{open}

\section{Mutual Independence of the Complexity Axioms}
\subsection{Normal Subgroup Property is Independent from the Other Complexity Axioms}\label{normal}

Here we construct complexity functions that do not have the normal subgroup property. 
Let $S$ be a finite simple group.  For each finite group, let 
  
$$\surS_S(G)
   = \begin{cases}
      1 &  \mbox{there is surjective morphism }\varphi: G \sur S \\
      0      & \mbox{ otherwise.}
    \end{cases}$$

    This is the characteristic function of groups mapping onto $S$.

Then $\surS_S(S)=1$ but $\surS_S(K)=0$ for $K$ any other finite simple group. 
Hence by Lemma~\ref{construct-simple}, $\surS_S(G)$ satisfies the \underline{constructability}  axiom.

$\surS_S(\1)=0$ so the \underline{initial condition} holds. 

The \underline{product axiom} holds:   Consider a product group $G\times H$. 
If $H \sur S$, then pre-composing with the projection yields $G \times H \sur H \sur S$.   
Thus $\surS_S(H)=1$ implies $\surS_S(G\times H)=1$.  Similarly, for the case $\surS_S(G)=1$. \\
This shows if $\max(\surS_S(G), \surS_S(H))=1$ then $\surS_S(G\times H)=1$.  
Otherwise,  $\max(\surS_S(G), \surS_S(H))=0$.  Suppose there were a  $\varphi: G\times H \sur S$.  Pre-composing with the injection
$$G\cong G \times \1 \rightarrow G \times H \sur S.$$
The image of $G\cong G\times \1$ in  $S$ is a normal subgroup of $S$, since a surjective morphism maps normal subgroups to normal subgroups of its image, thus $\varphi(G\times \1)$ is either $\1$ or $S$.  But the image cannot be $S$ since $\surS_S(G)=0$, i.e. $G$ does not map onto $S$.  Similarly, $\1\times H$ maps to $1$ since $\surS_S(H)=0$.  This shows $G\times \1$ and $\1 \times H$ are both in the kernel of $\varphi$, therefore $(G\times \1)(\1 \times H) = G\times H$ is in the kernel of $\varphi$, contradicting $\varphi(G\times H)=S$. This shows if $\max(\surS_S(G), \surS_S(H))=0$ then $\surS_S(G\times H)=0$. Thus, the
product axiom holds. 

The \underline{quotient axiom} holds: Suppose $G\sur H$.
If $\surS_S(G)=1$, then we are done since $\surS_S(G)\geq \surS_S(H)$.
Otherwise $\surS_S(G)=0$. Suppose there were a surjective morphism $H\sur K$.
The precomposing with the map from $G$ to $H$ yields a surjective morphism $G\sur H\sur K$, contradicting $\surS_S(G)=0$.
Hence there can be no map $H\sur K$ and $\surS_S(H)=0$.  In either case,
$\surS_S(G)\geq \surS_S(H)$. 

\underline{Extension axiom}: Suppose $N\lhd G$. We must show $\surS_S(G) \leq \surS_S(N)+\sur(G/N)$.
If $\surS_S(G)=0$, we are done. If $\surS_S(G)=1$, then we have a $\varphi: G\sur S$.
Let $K$ be the kernel of $\varphi$.  If $N\lhd K$, then $\varphi(N)=\1$. 
Otherwise, since surjective morphisms map normal subgroups to normal subgroups and $S$ is simple, it follows that $\varphi(N)=S$, 
whence $\surS_S(G) =1  \leq  1 + \surS_S(G/N) = \surS_S(N)+\surS_S(G/N)$ holds. 
We are left with the case $N \lhd K$. 
In this case, $G/N$ maps homomorphically onto $G/K$, which maps onto $S$. 
That is, $\surS_S(G/N)=1$ and extension holds. 

\begin{proposition}\label{surS-no-normal}
 The function $\surS_S$ satisfies the initial condition, constructability axiom, product axiom, quotient axiom and the extension axiom, hence is a complexity function on finite groups.  However, $\surS_S$ does not have the normal subgroup property, that is there  a group $G$ and normal subgroup $N$ with $\surS_S(N) > \surS_S(G)$.
\end{proposition}
 \begin{proof} $\surS_S$ is a complexity function since it satisfies the axioms listed as just shown.  Suppose $G$ is any group that is an 
 extension by $N=S$ of any other simple group $Q$ ($Q\not\cong S$), such  that $G$ is not a direct 
 product.  Then $G$ maps onto $Q$ but not $S$, so $\surS_S(G)=0$ but $\surS_S(S)=1$ even though $S\lhd G$.
 More generally, 
 we can always take $G$ to be the wreath product $S\wr Q$, which is a semidirect product of $Q$ with $N$, a direct product of copies of $S$. The wreath product $S\wr Q$ maps onto $Q$ but not $S$, and has $N$ as a normal subgroup.  
 By the product axiom,    $\surS_S(N)= 1$, but $\surS(G)=0$. 
 \end{proof}

\begin{counterexample}[Failure of  Normal Subgroup Property]
{\it The normal subgroup property need not hold for a complexity function.} 

\begin{enumerate}
 \item For the complexity function $\surS_{A_5}$, take  $G=S_5$ the symmetric group on 5 objects, which is an extension of $\Z_2$ by the alternating simple group $A_5$.  We have $\surS_{A_5}(S_5)=0$ but $\surS_{A_5}(A_5)=1$ with  the simple alternating group $A_5 \lhd S_5$.  
\item For the complexity function $\surS_{\Z_2}$, consider the special linear group $SL(2,5)$ of all $2\times 2$ matrices over the field with 5 elements having determinant 1.  
$SL(2,5)$ maps onto $PSL(2,5)\cong A_5$ with kernel isomorphic to $\Z_2 = \pm I$, where $I$ is the identity matrix. 
We have $\surS_{\Z_2}(SL(2,5))=0$ but $\Z_2 \lhd SL(2,5)$ with $\surS_{\Z_2}(\Z_2)=1$. 
\item (Smallest Counterexample.) The complexity function $\surS_{\Z_3}$ takes value 0 on $S_3$, the symmetric group on 3 objects, which is an extension of $\Z_2$ by $\Z_3$. But $\Z_3\lhd S_3$ and $\surS_{\Z_3}(\Z_3)=1$.
\end{enumerate}
\end{counterexample}

\subsection{Independence of the Quotient Axiom}

\begin{proposition}\label{quo-indep}
Let $\V$ be a class of finite groups closed under direct products and normal subgroups which is not the class of all finite groups and has a least one non-trivial member.  
 Then the {\em characteristic function of non-trivial members of $\V$} defined by 
$$\subS_\V(G)
   = \begin{cases}
      1 &  \mbox{if $\1 \neq G$ lies in $\V$}\\
      0      & \mbox{ otherwise.}
    \end{cases}$$
is a complexity function, but does  {\em not} satisfy the 
quotient axiom. 
\end{proposition}
\begin{proof} It is immediate from the properties and $\V$ and the definition of $\subS_\V$  that the initial, product   and normal subgroup axioms hold.  Constructability holds by Lemma~\ref{construct-simple}.

\underline{Extension axiom}: 
Suppose $N$ is normal in a finite group $G$. 
If $G$ is not in $\V$ or $G=\1$, then $0= \subS_\V(G) \leq \subS_\V(N)+ \subS_\V(G/N)$, i.e., the extension property  holds.
Otherwise,  $1\neq G$ is in $\V$,  
then either $N=\1$ and so $\subS_\V(G)= 1 = 0 + 1=  \subS_\V(N)+ \subS_\V(G/N),$ since $G/N\cong G$;  or 
$N\neq \1$ and so  $\subS_\V (G)=1 \leq 1+  \subS_\V(G/N) = \subS_\V(N) +  \subS_\V(G/N).$   
Thus $\subS_\V$ satisfies the extension axioms.   We conlude $\subS$ is a complexity function.

\underline{Failure of quotient property}:  By the hypotheses, $\V$ contains some nontrivial group $V$ and 
there exists a least one group $L$ not in $\V$.
Observe that  $V \times L$ is not in $\V$, since  $\V$ is closed under normal subgroups but its normal subgroup $L \cong \1 \times L$ is not in $\V$.  Now $V$ is a quotient of $V\times L$, but we have
$\subS_V(V\times L)=0$ and $\subS(V)=1$. 
Hence, the quotient axiom does not hold for $\subS_V$. 
\end{proof}

\begin{examples}
Using Proposition~\ref{quo-indep} we may construct a plethora of  complexity functions not satisfying the quotient axiom:
\begin{enumerate}
\item Consider  $\Span(\SIMPLE)$  the collection of all finite products of finite simple groups.
 By Lemma~\ref{SpanOfGems}, products, normal subgroups and quotients of members of  $\Span(\SIMPLE)$ lie in $\Span(\SIMPLE)$.  
By Prop.~\ref{quo-indep},  $\subS_{\Span(\SIMPLE)}$ satisfies all the complexity axioms except quotient axiom.
For example, this function assigns the value 1 to the quotient $\Z_2\cong S_5/A_5$ of the symmetric group $S_5$  but
assigns $0$ to $S_5$.
\item Let $\V=\Nil$, all nilpotent groups, or $\V=\Solv$, all solvable groups, and applying Prop.~\ref{quo-indep}, we obtain
that the characteristic function $\subS_{\Nil}$ for non-trivial nilpotent groups and the characteristic function $\subS_{\Solv}$  for non-trivial groups solvable groups are both complexity functions not satisfying the quotient axiom. 
\end{enumerate}
\end{examples}

\subsection{Independence of Each Complexity Axiom from the Rest}

\begin{theorem}[Independence] Consider the set of five complexity axioms  $$\Ax= \{ \mbox{\rm product axiom, extension axiom, initial , quotient, constructability}\},$$  (as they are stated in the Introduction)  and also the 
adding the normal subgroup property, consider $$\Ax'= \Ax \cup \{ \mbox{\rm normal subgroup property}\}.$$
Then for each $\phi \in \Ax'$, there is a function on $\c$ on finite groups satisfying all members of $\Ax'\setminus \{\phi\}$, but not satisfying $\phi$.
Therefore, none of the axioms follows from a proper subset of $\Ax'$. 
\end{theorem}
\begin{proof}

Here in the table below we collect example functions satisfying all but one property $\phi$ and references to proofs that this is so.
This proves the theorem.

\begin{tabular}{|l|c|l|}
\hline
Axiom / Property $\phi$ & Function & Proof \\
\hline
\hline
Product & $\JH$ & Theorem~\ref{JH}\\
\hline
Extension & $\sx$ & Corollary~\ref{spx-no-ext} \\
\hline
Initial & $\c=1$ & (trivial check)\\
\hline
Quotient &$\subS_\Nil$ & Prop.~\ref{quo-indep}\\
\hline
Constructability &$\c(G)=2$ for $G\neq\1$, &  Lemma~\ref{construct-simple}, \\
 & $\c(\1)=0$ & plus trivial check.\\
\hline
Normal Subgroup &  $\surS_S$ ($S$ simple) & Prop.~\ref{surS-no-normal}\\
\hline
\end{tabular}

\noindent \ \ \ {\bf Functions satisfying All Axioms/Properties except the one named.}

\end{proof}
\begin{remarks}

1. While the subgroup axiom is not one of the complexity axioms,
for completeness, we record here that $\cx$ is an example of a function that
satisfies satisfies all
the axioms and the normal subgroup property, but not the subgroup axiom (Prop.~\ref{cpx-no-subgroup}).

2. In addition, the following complexity functions on finite groups satisfy all the axioms including the normal subgroup property: $z$, $\delta$, $\cx$, $\chi_{\S}$, $\chi^*_{\S}$, $\Der$, $\Fit$, $\Solv$.

3. On solvable groups, the following satisfy all the axioms including the subgroup axiom:
$z$, $\delta$, $\cx=\cx_\solv$, $\chi_{\S}$,  $\chi^*_{\S}$, $\log_p$ ($p$ prime),  $\Der=\Der_\solv$, $\Fit=\Fit_\Solv$, $\Solv_\solv=\delta$.
\end{remarks}

\bibliographystyle{abbrv}
\bibliography{grpcpx}

\end{document}